\documentclass[]{article}

\usepackage[T1]{fontenc}
\usepackage{amsmath,amssymb,amsthm,mathrsfs}
\usepackage{pxfonts}
\usepackage{sectsty,microtype}
\usepackage[hidelinks]{hyperref}
\allowdisplaybreaks
\sectionfont{\scshape\centering\fontsize{11}{14}\selectfont}
\subsectionfont{\scshape\fontsize{11}{14}\selectfont}

\usepackage{fancyhdr}
\newcommand\shorttitle{BKT transition for \(\mathsf{Sine}_\beta\)}
\newcommand\authors{T. Assiotis}
\newtheorem{thm}{Theorem}[section]
\newtheorem{lem}[thm]{Lemma}
\newtheorem{prop}[thm]{Proposition}
\theoremstyle{definition}

\newtheorem{rmk}[thm]{Remark}
\theoremstyle{plain}

\newcommand{\E}{\mathbb E}
\newcommand{\Prob}{\mathbb P}
\newcommand{\R}{\mathbb R}
\newcommand{\C}{\mathbb C}
\newcommand{\N}{\mathbb N}
\newcommand{\Z}{\mathbb Z}
\newcommand{\Sine}{\mathsf{Sine}}
\newcommand{\dd}{\mathrm d}
\newcommand{\ii}{\mathrm i}
\newcommand{\ee}{\mathrm{e}}
\newcommand{\defeq}{\overset{\mathrm{def}}{=}}
\hypersetup{pdftitle={BKT transition for Sine beta},pdfauthor={Theodoros Assiotis}}

\title{\large\bfseries BKT TRANSITION FOR \(\Sine_\beta\)}
\author{\small THEODOROS ASSIOTIS}
\date{}
\begin{document}
\maketitle

\begin{abstract}
We prove the Berezinskii--Kosterlitz--Thouless (BKT) phase transition in the
large-distance behaviour of the two-point correlation function of the
$\Sine_\beta$ point process for all $\beta>0$, confirming the first term of an asymptotic expansion
predicted by Forrester. 
\end{abstract}

\section{Introduction}\label{sec:introduction}
\subsection{Setting and main result}

The \(\Sine_\beta\) point process describes the universal microscopic behaviour of eigenvalues in the bulk of \(\beta\)-ensembles. The quickest way to define it is via a limit transition as follows: for \(\beta>0\), let \(\mathsf T_1^{(N)},\ldots,\mathsf T_N^{(N)}\) be random angles with joint density
\[
\frac{1}{Z_{N,\beta}}
\prod_{1\le j<k\le N}
\left|\ee^{\mathrm i\theta_j}-\ee^{\mathrm i\theta_k}\right|^\beta
\prod_{j=1}^N\frac{\mathrm d\theta_j}{2\pi},
\qquad
\left(\theta_1,\ldots,\theta_N\right)\in\left[-\pi,\pi\right)^N,
\]
where \(Z_{N,\beta}\) is the normalising constant. Then, we have
\[
\sum_{j=1}^N\delta_{N\mathsf T_j^{\left(N\right)}/\left(2\pi\right)}
\ \overset{\mathrm w}{\longrightarrow}\ \mathsf X_\beta
\]
where the convergence is in distribution for the vague topology and \(\mathsf X_\beta\) is a stationary simple point process of intensity one, called \(\Sine_\beta\), see the work of Killip--Stoiciu \cite{KillipStoiciu} for this limit result.

A construction at the same time as \cite{KillipStoiciu} through the Gaussian \(\beta\)-ensemble and the Brownian carousel is due to Valk\'o--Vir\'ag \cite{ValkoVirag}, building on the random matrix models of Dumitriu--Edelman \cite{DumitriuEdelman}. Universality for a broad class of confining potentials was established by Bourgade--Erd\H{o}s--Yau \cite{BourgadeErdosYau}. Complementary descriptions include its realisation as the spectrum of a random Dirac operator by Valk\'o--Vir\'ag \cite{ValkoViragOperator}, its canonical Dobrushin--Lanford--Ruelle equations proved by Dereudre--Hardy--Lebl\'e--Ma\"ida \cite{DereudreHardyLebleMaida}, and its characterisation as the unique stationary minimiser of the infinite-volume log-gas free energy by Erbar--Huesmann--Lebl\'e \cite{ErbarHuesmannLeble}, see also more recently \cite{AssiotisDLR}.

Understanding the correlation functions of \(\Sine_\beta\) is a longstanding problem. At the classical values \(\beta=1,2,4\), they admit explicit Pfaffian or determinantal formulas; in particular, \(\Sine_2\) is the determinantal process with the sine kernel. We refer to Forrester's monograph \cite{ForresterBook} for these formulas and their random matrix origins. For general \(\beta\), useful representations and precise asymptotic information have proved substantially harder to obtain.

A central question concerns the two-point correlation function at large separation. In 1984, more than two decades before \(\Sine_\beta\)'s rigorous construction, in seminal work Forrester \cite{Forrester1984} predicted a change in its leading asymptotics at \(\beta=2\), through the connection with the Calogero--Sutherland model and Haldane's theory \cite{Haldane} of one-dimensional quantum fluids. At unit intensity, the predicted leading truncated correlation is proportional to \(-r^{-2}\) for \(0<\beta<2\), and to \(r^{-4/\beta}\cos(2\pi r)\) for \(\beta>2\), with the two contributions having equal order at \(\beta=2\). For more about this prediction and allied conjectures see Lewin's survey \cite{Lewin}. This change is referred to as the Berezinskii--Kosterlitz--Thouless (BKT) phase transition for the one-dimensional log-gas, following the seminal work of Berezinskii \cite{Berezinskii} and Kosterlitz--Thouless \cite{KosterlitzThouless} on low-dimensional systems. 

Forrester subsequently developed much of the exact framework underlying this prediction. His works on Selberg integrals, generalised hypergeometric functions and Jack polynomials give correlation formulas at even integer \(\beta\) and access to their asymptotics, see \cite{Forrester1992,Forrester1993,Forrester1994,Forrester1995}. Connections with differential equations were developed further by Forrester--Rains \cite{ForresterRains}, Forrester--Rahman \cite{ForresterRahman} and more recently by Forrester in  \cite{Forrester2021}, see also Chapter~13 of his monograph \cite{ForresterBook}. At rational coupling, important exact formulas were obtained by Ha \cite{Ha} for the Calogero--Sutherland model and by Okounkov \cite{Okounkov} for general log-gas correlation functions, with related integral evaluations due to Forrester--Zuk \cite{ForresterZuk}. 

A probabilistic representation of the pair correlation valid for arbitrary \(\beta>0\) became available only in 2025, in breakthrough work of Qu and Valk\'o \cite{QuValko}. Their formula, expressed through the Hua--Pickrell point process counting function, yields quantitative pointwise decay bounds for all \(\beta>0\), attaining the predicted decay exponent in the upper bound for \(\beta\ge4\). Importantly, it allows for the first time the use of probabilistic and stochastic analytic techniques in studying the two point correlation function, techniques which we will also develop in a different direction in this paper. In a complementary and independent advance, Dumaz--Malvy \cite{DumazMalvy} prove polynomial decay of spatially averaged truncated correlations of every order and for every \(\beta>0\). Finally, at the special parameter $\beta=6$, Qiu--Qu--Yuan--Valk\'o--Venancio \cite{QiuEtAl} recently obtained a completely explicit formula and refined asymptotics for the pair correlation of \(\Sine_6\). In recent joint work with Najnudel, the author obtained a formula for all correlation functions, for every \(\beta>0\), in terms of mixed absolute moments of the Hua--Pickrell stochastic zeta function, see \cite{AssiotisNajnudel}. This random entire function was introduced by Li--Valk\'o \cite{LiValko}, extending the stochastic-zeta framework of Valk\'o--Vir\'ag \cite{ValkoViragZeta}. An earlier construction and analysis in the determinantal case were given by Chhaibi--Najnudel--Nikeghbali \cite{ChhaibiNajnudelNikeghbali}. The formula of \cite{AssiotisNajnudel} separates an explicit Vandermonde factor from a stochastic-zeta expectation that remains continuous and strictly positive at coincident arguments. It is therefore particularly well adapted to fusion asymptotics, where particles approach one another, and immediately gives their leading behaviour. This line of research was further developed by Fang who obtained higher-order, critical and subcritical fusion asymptotics \cite{FangSecond,FangCritical,FangFourth}.

We now state our result precisely. The continuous factorial two-point correlation function \(\boldsymbol\rho_\beta^{(2)}\) (see \cite[Theorem~1]{QuValko} and \cite[Theorem~1.8]{AssiotisNajnudel}) is characterised by
\[
\mathbb E\left[
\sum_{\substack{u,v\in\mathsf X_\beta\\u\ne v}}
f\left(u,v\right)
\right]
=
\int_{\mathbb R^2}
f\left(u,v\right)\boldsymbol\rho_\beta^{\left(2\right)}\left(u,v\right)
\,\mathrm du\,\mathrm dv,
\qquad
0\le f\in\mathsf C_c\left(\mathbb R^2\right).
\]
The truncated two-point correlation function is then defined by
\[
\mathfrak h_\beta\left(r\right)
\defeq
\boldsymbol\rho_\beta^{\left(2\right)}\left(0,r\right)-1.
\]
Our main result determines the leading large-distance asymptotics on both sides of \(\beta=2\), including the exact oscillatory coefficient.

\begin{thm}\label{thm:main} For every fixed \(\beta>0\), as \(r\to+\infty\),
\begin{equation}\label{eq:main}
\mathfrak h_\beta\left(r\right)
=
\left\{
\begin{array}{ll}
-\displaystyle\frac{1}{\pi^2\beta r^2}
+o\left(r^{-2}\right),
&0<\beta<2,\\[8pt]
\displaystyle\frac{\cos\left(2\pi r\right)-1}{2\pi^2r^2},
&\beta=2,\\[8pt]
\displaystyle
\frac{2\Gamma\left(1+2/\beta\right)^2}{\left(2\pi\right)^{4/\beta}}
\frac{\cos\left(2\pi r\right)}{r^{4/\beta}}
+o\left(r^{-4/\beta}\right),
&\beta>2.
\end{array}
\right.
\end{equation}

\end{thm}

Obtaining a complete large-distance expansion of the second correlation function for arbitrary \(\beta>0\) (of course for even integer $\beta$ this and much more has already been proven in the works of Forrester), or precise pointwise asymptotics for higher correlation functions, remains open and very interesting. We believe that developing the asymptotic methods presented here could provide a starting point.

Our proof analyses the absolutely convergent series representation
of $\mathfrak h_\beta(r)$ obtained by Qu--Valk\'o
\cite{QuValko}, whose Fourier moments satisfy an infinite coupled
system of ordinary differential equations.
A related moment formulation follows from the author's formula
with Najnudel \cite{AssiotisNajnudel}, combined with
Li--Valk\'o's SDE representation of the Hua--Pickrell stochastic
zeta function \cite{LiValko}, and can be studied by the same
asymptotic method.\footnote{With the notation of \eqref{eq:normalisation},
\cite[Theorem~1.8]{AssiotisNajnudel} gives the first equality below:
\[
\begin{aligned}
\mathcal R_a(x)
&=\mathfrak c_a x^{2a}
  \E\left[\left|\mathsf Z^{2a,2a}(x)\right|^{2a}\right]=\mathfrak c_a   \left(\displaystyle\sum_{n\in\Z}t_{|n|}(a)^2\right)^{-1}
 \displaystyle\sum_{n\in\Z}
                  t_{|n|}(a)^2\mathcal P_n(x;a).
\end{aligned}
\]
Here $\mathsf Z^{2a,2a}$ is the Hua--Pickrell stochastic zeta
function, normalised by $\mathsf Z^{2a,2a}(0)=1$,
$\mathfrak c_a=a^{2a}\Gamma(a+1)^3/
[\Gamma(2a+1)\Gamma(3a+1)]$, and $t_n(a)$ is given by
\eqref{eq:correlation-weights}, with $t_0(a)=1$.
For the second equality, \cite[Theorem~22]{LiValko}, with
$\beta=\delta=2a$, gives in distribution
$\mathsf Z^{2a,2a}{=}
\mathsf A_0-\mathsf Q\mathsf B_0$,
where $(\mathsf A_u,\mathsf B_u)$ solves the SDE there and
$\mathsf Q$ is an independent real random variable with density
proportional to $(1+q^2)^{-2a-1}$.
Write
$\mathsf A_0(x)-\ii\mathsf B_0(x)
=\mathsf\rho(x)\ee^{\ii\mathsf\vartheta(x)/2}$ and put
$\mathcal P_n(x;a)= x^{2a}
\E[\mathsf\rho(x)^{2a}\ee^{\ii n\mathsf\vartheta(x)}]$.
Averaging over $\mathsf Q$ gives the stated Fourier series,
while It\^o's formula gives
\[
\partial_x\mathcal P_n
=\left(\ii n-\frac{2n^2}{ax}\right)\mathcal P_n
+\frac{(n+a)^2}{ax}\mathcal P_{n-1}
+\frac{(n-a)^2}{ax}\mathcal P_{n+1},
\qquad n\in\Z.
\]
This system has the same diagonal oscillation and damping as
\eqref{eq:q-hierarchy}, and the subsequent steps of our
asymptotic analysis can be adapted to it.}
We introduce an explicit correction that removes the
off-diagonal interactions of order $x^{-1}$ while retaining
the diagonal damping, which separates the decay rates of
the different Fourier modes.
The principal probabilistic input is a uniform decay estimate
in the Fourier index which provides the control of
large indices needed to justify the analysis of the infinite
system and the summation of its asymptotics.
For $\beta>2$, the resulting expansion contains an initially
unknown oscillatory coefficient. We determine it by holomorphic continuation, starting from its
values at even integer $\beta$, supplied by Forrester's asymptotics
\cite{ForresterBook}. The identification of the Selberg expression
with the pair correlation follows from the author's work with
Najnudel \cite[Proposition~2.15 and (21)]{AssiotisNajnudel}.
We give a more explicit outline in the next subsection.

\subsection{On the proof}\label{sec:strategy}

We use the normalisation
\begin{equation}\label{eq:normalisation}
 a\defeq\frac\beta2,\qquad p\defeq\frac2a,\qquad
 \mathcal R_a\left(x\right)\defeq
 1+\mathfrak h_{2a}\left(\frac{x}{2\pi}\right).
\end{equation}
 $\mathcal R_a$ is expressed in the work of Qu-Valk\'{o} \cite{QuValko}  as an absolutely convergent sum of
Fourier moments $\mathfrak q_n(x;a)$ of a certain Hua--Pickrell point process phase diffusion,
\begin{equation*}
 \mathcal R_a\left(x\right)
 =1+2\sum_{n=1}^\infty
       t_n\left(a\right)\Re\mathfrak q_n\left(x;a\right).
\end{equation*}
and \(t_n(a)\) is an explicit combinatorial factor defined in Section~\ref{sec:representation}.
The differential system the $\mathfrak q_n(x;a)$ satisfy which is key for our analysis is given by 
\begin{equation*}
 \partial_x\mathfrak q_n\left(x;a\right)
 =\left(\ii n-\frac{pn^2}{x}\right)\mathfrak q_n\left(x;a\right)
   +\frac{n\left(n+a\right)}{ax}\mathfrak q_{n-1}\left(x;a\right)
   +\frac{n\left(n-a\right)}{ax}\mathfrak q_{n+1}\left(x;a\right),
 \quad n\in\Z.
\end{equation*}
while their probabilistic representation is given in
Section~\ref{sec:representation}. 

Our first
task is to obtain the following expansion, with an initially unknown
oscillatory coefficient $\mathfrak U(a)$:
\[
 \mathcal R_a\left(x\right)=
 \begin{cases}
  1-px^{-2}+o\left(x^{-2}\right),&0<a<1,\\[3pt]
  1+2\Re\left(\mathfrak U\left(a\right)\ee^{\ii x}\right)x^{-p}
       +o\left(x^{-p}\right),&a>1.
 \end{cases}
\]
This is done in Proposition~\ref{prop:real-asymptotics}. Let us explain how.

The key input from the probabilistic representation of $\mathfrak q_n$ is in the following. The moment equations have diagonal terms $\ii n-pn^2/x$ and interactions
with neighbouring indices of order $x^{-1}$. To control the infinite
sum, Lemma~\ref{lem:uniform-modes} uses repeated Gaussian integration
by parts along a deterministic shift of one driving Brownian motion
to prove, for every $s\ge0$, the uniform estimate which will be used as seed in many subsequent estimates
\[
 \sup_{x\ge1}\sup_{n\in\Z}
 \left(1+\left|n\right|\right)^s
 \left|\mathfrak q_n\left(x;a\right)\right|<\infty.
\]

With this key estimate at hand we then analyse asymptotically the differential equations directly with elementary methods. We first subtract a finite algebraic approximation to the moments.
Writing $\boldsymbol{\mathfrak q}$ for the vector with coordinates
$\mathfrak q_n$, set
\[
 \mathbf Q\left(x;a\right)\defeq
 \mathbf e_0+\frac{\mathbf s_1\left(a\right)}x
 +\frac{\mathbf s_2\left(a\right)}{x^2}
 +\frac{\mathbf s_3\left(a\right)}{x^3},
 \qquad
 \boldsymbol{\mathfrak r}\defeq
 \boldsymbol{\mathfrak q}-\mathbf Q.
\]
Here $\mathbf e_0$ is the unit vector at index zero, and the
finitely supported coefficients $\mathbf s_j$ are determined
successively from the hierarchy. The residual $\mathfrak r$ has zeroth coordinate
zero and satisfies the same homogeneous system with an explicit
forcing $\mathbf g(a)/x^4$, where $\mathbf g(a)$ has finite support.
The third-order term is included to obtain this order of forcing but
only $\mathbf s_1$ and $\mathbf s_2$ need to actually be computed.

We next correct the residual to remove its leading order interaction.
In vector form the correction satisfies
\[
 \begin{gathered}
 \boldsymbol{\mathfrak w}\defeq\boldsymbol{\mathfrak r}
 +\frac{\mathbf T\left(a\right)\boldsymbol{\mathfrak r}}x,
 \\
 \boldsymbol{\mathfrak w}'
 =\left(\ii\mathcal N-\frac{p\mathcal N^2}{x}\right)
       \boldsymbol{\mathfrak w}
 +\frac{\mathbf F\left(a\right)\boldsymbol{\mathfrak r}}{x^2}
 +\frac{\mathbf g\left(a\right)}{x^4}
 +\frac{\mathbf T\left(a\right)\mathbf g\left(a\right)}{x^5}.
 \end{gathered}
\]
Here $\mathcal N$ multiplies the $n$th coordinate by $n$, and the
matrices $\mathbf T$ and $\mathbf F$ involve only finitely many
neighbours in each row. Their coefficients grow at most quadratically
and quartically in $n$, respectively. The scalar solution factor is
\[
 \ee^{\ii n\left(x-t\right)}
 \left(\frac tx\right)^{pn^2},\qquad 1\le t\le x.
\]
Its modulus $(t/x)^{pn^2}$ describes the decay of the
homogeneous solutions, with faster decay for larger mode indices.

The distinction between $p<2$ and $p>2$ occurs in this corrected
equation. When $p<2$, variation of constants, together with the uniform
bounds above, first gives $\mathfrak r_n=O(x^{-p/2})$
with every fixed polynomial weight in the mode index. The first-mode integrating factor then produces
an absolutely integrable forcing. Consequently
$\ee^{-\ii x}x^p\mathfrak w_1(x;a)$ converges, with limit
$\mathfrak U(a)/t_1(a)$. Recovering $\mathfrak q_1$ adds a purely
imaginary term of order $x^{-1}$, which vanishes from the real
correlation, and smaller terms. Higher modes are negligible at
order $x^{-p}$. When $p>2$, every nonzero homogeneous mode decays
faster than $x^{-2}$. Successive damping estimates improve the
residual bound from $O(x^{-1})$ to $O(x^{-2})$ and then to
$o(x^{-2})$, uniformly in the mode index. This gives the
second-order expansion with coefficients $\mathbf s_1$ and
$\mathbf s_2$. The first coefficient is purely imaginary.
Only the first two positive modes contribute to the coefficient
$-p$, and the uniform remainder can be summed against the
correlation weights.

It remains to identify the oscillatory coefficient. For integers
$m\ge4$, the convergence and uniform bound in
\cite[Proposition~2.15 and (21)]{AssiotisNajnudel} identify
Forrester's Selberg integral with the pair correlation.
Forrester's asymptotics
\cite[Proposition~13.2.4, pp.~604--605]{ForresterBook} then give
\[
 \mathfrak U\left(m\right)=\Gamma\left(1+1/m\right)^2.
\]
We extend the identity from integer parameters to all real $a>1$
by proving holomorphic continuation and polynomial growth of
$\psi_0(a)\mathfrak U(a)$ in Lemma~\ref{lem:amplitude-continuation}.
Here $\psi_n(a)$ are the Fourier coefficients of
$|\sin(\theta/2)|^{2a}$. Weighting the moments by $\psi_n(a)$
allows us to control finite approximations of their hierarchy. These bounds make the
first-mode amplitude converge locally uniformly in the complex
parameter and yield the required polynomial growth.
After division by a sufficiently large power of $a+1$, the
difference from $\psi_0(a)\Gamma(1+1/a)^2$ is bounded and
holomorphic on $\Re a>3$. Its zeros at every integer $m\ge4$
violate the Blaschke condition unless it vanishes identically.
Analytic continuation therefore gives
$\mathfrak U(a)=\Gamma(1+1/a)^2$ for all real $a>1$.
Together with Proposition~\ref{prop:real-asymptotics} and the
sine-kernel formula at $a=1$, this proves the main theorem.

\paragraph{AI use statement} AI tools (ChatGPT 5.6 Sol, 6) assisted the technical implementation of the proof extensively following a general strategy and ideas of the author. The individual proofs of intermediate results have, in one case (Lemma \ref{lem:uniform-modes}) been reworked completely  and, in general have all been verified, simplified and modified for clarity by the author over the past months. AI tools (ChatGPT 6) have also been used for literature review, proofreading and general editing of the manuscript.

\section{The Hua--Pickrell phase representation}\label{sec:representation}

Throughout this section $a>0$ and $x>0$. We recall the phase
representation and moment equations of Qu--Valk\'o, which are the
starting point of the proof. Let
$\mathsf W_{\C}=\mathsf W_1+\ii\mathsf W_2$, where $\mathsf W_1$ and
$\mathsf W_2$ are independent two-sided standard Brownian motions.
Consider the strong solution with entrance from $0$ at time $-\infty$
\begin{equation}\label{eq:phase-sde}
 \begin{aligned}
 \dd\mathsf\alpha_u\left(x;a\right)
 &=\left(\frac{ax}{2}\ee^{au/2}
              -a\sin\mathsf\alpha_u\left(x;a\right)\right)\dd u
   +\Re\left[\left(\ee^{-\ii\mathsf\alpha_u\left(x;a\right)}-1\right)
                 \dd\mathsf W_{\C}\left(u\right)\right],\\
 \lim_{u\to-\infty}\mathsf\alpha_u\left(x;a\right)&=0.
 \end{aligned}
\end{equation}
The equation is interpreted forward on each finite interval in
$(-\infty,0]$, with the filtration generated by past Brownian
increments. Existence and uniqueness of this  solution follow
from \cite[Proposition~6 and Theorem~1]{QuValko}; see also
\cite{LiValko} for the underlying Hua--Pickrell construction.
Define its Fourier moments by
\begin{equation}\label{eq:phase-moments}
 \mathfrak q_n\left(x;a\right)\defeq
 \E\left[\ee^{\ii n\mathsf\alpha_0\left(x;a\right)}\right],
 \qquad n\in\Z.
\end{equation}
For $n\ge1$, let $(z)_n=\prod_{j=0}^{n-1}(z+j)$ be the rising factorial,
with $(z)_0=1$, and put
\begin{equation}\label{eq:correlation-weights}
 t_n\left(a\right)\defeq\frac{\left(-a\right)_n}{\left(1+a\right)_n}.
\end{equation}
Qu--Valk\'o's formula \cite[Theorem~1]{QuValko}, expressed at unit
intensity, is
\begin{equation}\label{eq:correlation-mode-sum}
 \mathcal R_a\left(x\right)
 =1+2\sum_{n=1}^\infty
       t_n\left(a\right)\Re\mathfrak q_n\left(x;a\right).
\end{equation}
Their intensity is $1/(2\pi)$, so their two-point function at $(0,x)$
is $\mathcal R_a(x)/(4\pi^2)$. The bound
$|t_n(a)|\le C_a(1+n)^{-2a-1}$ makes the sum absolutely convergent.
At an integer $a=m\ge1$, its weights vanish for $n>m$.

The elementary moment bounds and entrance values are
\begin{equation}\label{eq:physical-mode-bounds}
 \mathfrak q_0\left(x;a\right)=1,\qquad
 \left|\mathfrak q_n\left(x;a\right)\right|\le1,\qquad
 \mathfrak q_{-n}\left(x;a\right)
       =\overline{\mathfrak q_n\left(x;a\right)},\qquad
 \lim_{x\downarrow0}\mathfrak q_n\left(x;a\right)=1.
\end{equation}
The moment equations of \cite[Proposition~20]{QuValko}, specialised
to $\beta=2a$ and the Hua--Pickrell parameter $\delta=a$, read
\begin{equation}\label{eq:q-hierarchy}
 \partial_x\mathfrak q_n\left(x;a\right)
 =\left(\ii n-\frac{pn^2}{x}\right)\mathfrak q_n\left(x;a\right)
   +\frac{n\left(n+a\right)}{ax}\mathfrak q_{n-1}\left(x;a\right)
   +\frac{n\left(n-a\right)}{ax}\mathfrak q_{n+1}\left(x;a\right),
 \quad n\in\Z.
\end{equation}
They follow by applying It\^o's formula to the angular exponential and
using the scaling in law
$$\mathsf\alpha_u(x;a)\stackrel{\mathrm d}{=}
\mathsf\alpha_0(x\ee^{au/2};a).$$
In particular, the zero-coordinate equation is identically zero.
We analyse this hierarchy for real parameters in the next section.

\section{Large-distance asymptotics for real parameters}\label{sec:real-asymptotics}
\subsection{A uniform Fourier bound}

We first obtain the decay in the mode index needed to justify the changes of variables and the summations in the asymptotic analysis.
\begin{lem}\label{lem:uniform-modes}
For every \(a>0\) and \(s\ge0\),
\[
\sup_{x\ge1}\sup_{n\in\Z}\left(1+\left|n\right|\right)^s
\left|\mathfrak q_n\left(x;a\right)\right|<\infty.
\]
\end{lem}

\begin{proof}
The main step of the proof is the repeated Gaussian integration by
parts at its end, which produces a factor $n^{-1}$ at each iteration
and hence arbitrary decay in the mode index. The preceding estimates
provide the uniform moment bounds needed to justify this argument.
We perturb one of the driving Brownian motions in a fixed direction
and control the resulting derivatives of the terminal phase, together
with inverse moments of its first derivative. The key intuition is
that, although the diffusion coefficients vanish at multiples of
$2\pi$, the deterministic forcing, for $x\ge1$, prevents the phase
from spending too long near these points. Moreover, the contributions
to the first derivative under our chosen perturbation have the same
sign and therefore do not cancel. These two facts give the
quantitative non-degeneracy needed for integration by parts, with
bounds uniform in $x\ge1$.

\textbf{Step 1. Reduction to a finite interval and an occupation estimate.} Fix $a>0$, $s\ge0$, and an integer $m\ge\max(1,s)$.
We work on the canonical space
\[
\Omega\defeq
\left\{w\in C\left([-1,0];\R^2\right):
w\left(-1\right)=0\right\},
\]
equipped with two-dimensional Wiener measure $\Prob$ and expectation
$\E$. Its coordinate processes $\mathsf W_1,\mathsf W_2$ are independent
Brownian motions started from zero at time $-1$. They represent the
Brownian increments on the last interval $[-1,0]$ of \eqref{eq:phase-sde}.
For $x\ge1$ and $\lambda,\theta\in\R$, let
$\mathsf\alpha_u^{\lambda,\theta}(x;a)$ be the strong solution of
\begin{equation}\label{eq:shifted-phase-sde}
\begin{aligned}
\dd\mathsf\alpha_u^{\lambda,\theta}
={}&\left[
\frac{ax}{2}\ee^{au/2}
-a\sin\left(\mathsf\alpha_u^{\lambda,\theta}\right)
+\lambda\left(\cos\left(\mathsf\alpha_u^{\lambda,\theta}\right)-1\right)
\right]\dd u
\\
&+\left(\cos\left(\mathsf\alpha_u^{\lambda,\theta}\right)-1\right)
\dd\mathsf W_{1,u}
+\sin\left(\mathsf\alpha_u^{\lambda,\theta}\right)\dd\mathsf W_{2,u},
\qquad \mathsf\alpha_{-1}^{\lambda,\theta}=\theta.
\end{aligned}
\end{equation}
At $\lambda=0$ this is the phase evolution \eqref{eq:phase-sde} started from $\theta$ at
time $-1$. Put
\[
\mathfrak K_{x,n}\left(\theta;a\right)\defeq
\E\left[\ee^{\ii n\mathsf\alpha_0^{0,\theta}\left(x;a\right)}\right].
\]
The Markov property of the entrance solution in \eqref{eq:phase-sde}
gives
\[
\mathfrak q_n\left(x;a\right)
=\E\left[
\mathfrak K_{x,n}\left(\mathsf\alpha_{-1}\left(x;a\right);a\right)
\right],
\]
where the expectation in this display is under the law of the
entrance solution. It suffices to prove
$|\mathfrak K_{x,n}(\theta;a)|\le C_{a,m}|n|^{-m}$ uniformly in
$x\ge1$ and $\theta\in\R$, for $n\ne0$.
In the local calculations below we write $\mathsf\alpha_u^\lambda$
for $\mathsf\alpha_u^{\lambda,\theta}(x;a)$, suppressing $x,a,\theta$.
Unless otherwise stated, bounds are uniform for $x\ge1$,
$\theta\in\R$ and $|\lambda|\le1$; constants may depend on $a$
and on the indicated moment and derivative orders.

We first control the integral
\[
\mathsf I_\lambda\defeq
\int_{-1}^0
\left[1-\cos\left(\mathsf\alpha_u^\lambda\right)\right]\dd u.
\]
The deterministic phase
\[
\vartheta_u\defeq
\theta+x\left(\ee^{au/2}-\ee^{-a/2}\right)
\]
satisfies
\begin{equation}\label{eq:deterministic-phase-rotation}
d_a\defeq
\inf_{x\ge1,\,\theta\in\R}
\int_{-1}^0\left[1-\cos\left(\vartheta_u\right)\right]\dd u>0.
\end{equation}
Indeed, every integral is positive, since $\vartheta$ is strictly
increasing. Continuity and periodicity in $\theta$ give a positive
lower bound when $x$ ranges over a compact interval. For large $x$,
the change of variables $t=\ee^{au/2}$ and integration by parts give
\[
\left|\int_{-1}^0\ee^{\ii x\ee^{au/2}}\,\dd u\right|
=\frac2a\left|\int_{\ee^{-a/2}}^1\frac{\ee^{\ii xt}}t\,\dd t\right|
\le\frac{4\ee^{a/2}}{ax}.
\]
Thus the integral of $1-\cos\vartheta_u$ is at least $1/2$ for all
sufficiently large $x$, uniformly in $\theta$.

The martingale
\[
\mathsf M_u^\lambda\defeq
\int_{-1}^u\left(\cos\left(\mathsf\alpha_v^\lambda\right)-1\right)
\dd\mathsf W_{1,v}
+\int_{-1}^u\sin\left(\mathsf\alpha_v^\lambda\right)
\dd\mathsf W_{2,v}
\]
starts from zero at time $-1$ and has quadratic variation
$[\mathsf M^\lambda]_0=2\mathsf I_\lambda$.
Since $\sin^2 y\le2(1-\cos y)$, the integral equation gives
\[
\sup_{-1\le u\le0}\left|\mathsf\alpha_u^\lambda-\vartheta_u\right|
\le a\sqrt{2\mathsf I_\lambda}
+\left|\lambda\right|\mathsf I_\lambda
+\sup_{-1\le u\le0}\left|\mathsf M_u^\lambda\right|.
\]
The function $1-\cos$ is Lipschitz with constant one. Hence
\[
d_a\le
2\mathsf I_\lambda+a\sqrt{2\mathsf I_\lambda}
+\sup_{-1\le u\le0}\left|\mathsf M_u^\lambda\right|.
\]
For sufficiently small $\varepsilon>0$, depending only on $a$,
the event $\{\mathsf I_\lambda\le\varepsilon\}$ therefore implies
$\sup_{-1\le u\le0}|\mathsf M_u^\lambda|\ge d_a/2$ and the quadratic variation 
$[\mathsf M^\lambda]_0\le2\varepsilon$.
For a continuous local martingale $\mathsf M$ started from zero
at time $-1$, the exponential martingale inequality reads
\[
\Prob\left(
\sup_{-1\le u\le0}\left|\mathsf M_u\right|\ge A,\,
[\mathsf M]_0\le B
\right)\le2\ee^{-A^2/(2B)},\qquad A,B>0.
\]
It follows that
\begin{equation}\label{eq:phase-variance-small-ball}
\sup_{\substack{x\ge1,\,\theta\in\R\\\left|\lambda\right|\le1}}
\Prob\left(\mathsf I_\lambda\le\varepsilon\right)
\le2\ee^{-d_a^2/\left(16\varepsilon\right)}.
\end{equation}
Consequently $\mathsf I_\lambda>0$ almost surely and, for every
$r>0$, $\E[\mathsf I_\lambda^{-r}]$ is bounded uniformly in these
parameters. The same argument works on any fixed interval
$|\lambda|\le R$, with $2\mathsf I_\lambda$ replaced by
$(1+R)\mathsf I_\lambda$ in the preceding deterministic bound.

\textbf{Step 2. Parameter derivatives and inverse moment bounds.} We use the standard smooth dependence of SDE solutions on their
initial values; see \cite[Theorem~1.4.1]{KunitaFlows1986}.
For each fixed $x$, apply this result to the augmented equation
with an additional coordinate $\Lambda_u$ satisfying
$\dd\Lambda_u=0$, $\Lambda_{-1}=\lambda$.
On any bounded interval of initial values of $\Lambda$, replace
$\Lambda$ in the drift by a smooth function of compact support
equal to it on a neighbourhood of that interval. The resulting
coefficients have bounded spatial derivatives of every order.
Since $\Lambda$ is constant, this replacement does not change the
solutions under consideration. The solution therefore has a version
smooth in $(\theta,\lambda)$, with derivatives continuous jointly
with $u$. Localizing by the suprema of these derivatives on compact
sets of initial values permits differentiation under the stochastic
integrals, giving the equations below.

Let $\mathsf J_u^\lambda=\partial_\theta\mathsf\alpha_u^\lambda$
be the derivative with respect to the initial phase. It solves
\[
\begin{aligned}
\dd\mathsf J_u^\lambda
={}&\left[-a\cos\left(\mathsf\alpha_u^\lambda\right)
-\lambda\sin\left(\mathsf\alpha_u^\lambda\right)\right]
\mathsf J_u^\lambda\,\dd u
-\sin\left(\mathsf\alpha_u^\lambda\right)
\mathsf J_u^\lambda\,\dd\mathsf W_{1,u}
+\cos\left(\mathsf\alpha_u^\lambda\right)
\mathsf J_u^\lambda\,\dd\mathsf W_{2,u},
\qquad \mathsf J_{-1}^\lambda=1.
\end{aligned}
\]
This scalar linear equation gives a strictly positive solution, with
\[
\begin{aligned}
\log\mathsf J_u^\lambda
={}&\int_{-1}^u\left[
-a\cos\left(\mathsf\alpha_v^\lambda\right)
-\lambda\sin\left(\mathsf\alpha_v^\lambda\right)-\frac12
\right]\dd v
-\int_{-1}^u\sin\left(\mathsf\alpha_v^\lambda\right)
\dd\mathsf W_{1,v}
+\int_{-1}^u\cos\left(\mathsf\alpha_v^\lambda\right)
\dd\mathsf W_{2,v}.
\end{aligned}
\]
The drift is bounded in absolute value by $a+3/2$, or by
$a+R+1/2$ on $|\lambda|\le R$, and the martingale in this
expression has quadratic variation $u+1$.
The exponential martingale inequality therefore implies that
\[
\mathsf J_*^\lambda\defeq
\sup_{-1\le u\le0}
\max\left(\mathsf J_u^\lambda,
\left(\mathsf J_u^\lambda\right)^{-1}\right)
\]
has moments of every positive order, uniformly in $x,\theta,\lambda$.

For the parameter derivatives, write
\[
\begin{gathered}
b_u^\lambda\left(y\right)\defeq
\frac{ax}{2}\ee^{au/2}-a\sin y+\lambda\left(\cos y-1\right),
\qquad
c_1\left(y\right)\defeq\cos y-1,\quad
c_2\left(y\right)\defeq\sin y,
\\
\mathsf L_u^\lambda\defeq\partial_\lambda\mathsf\alpha_u^\lambda,
\qquad
\mathsf F_\lambda\defeq\mathsf\alpha_0^\lambda,
\qquad
\mathsf L_\lambda\defeq\mathsf L_0^\lambda
=\partial_\lambda\mathsf F_\lambda.
\end{gathered}
\]
Thus $\mathsf F_\lambda$ is the terminal phase and
$\mathsf L_\lambda$ its parameter derivative. Primes on
$b_u^\lambda,c_1,c_2$ denote derivatives in their spatial argument.
The process $\mathsf L^\lambda$ satisfies
\[
\begin{aligned}
\dd\mathsf L_u^\lambda
={}&\left[
\left(b_u^\lambda\right)'\left(\mathsf\alpha_u^\lambda\right)
\mathsf L_u^\lambda+c_1\left(\mathsf\alpha_u^\lambda\right)
\right]\dd u
+\sum_{j=1}^2c_j'\left(\mathsf\alpha_u^\lambda\right)
\mathsf L_u^\lambda\,\dd\mathsf W_{j,u},
\qquad \mathsf L_{-1}^\lambda=0.
\end{aligned}
\]
Variation of constants, using the same homogeneous equation as for
$\mathsf J^\lambda$, gives
\[
\mathsf L_\lambda
=-\mathsf J_0^\lambda\int_{-1}^0
\frac{1-\cos\left(\mathsf\alpha_v^\lambda\right)}
{\mathsf J_v^\lambda}\,\dd v,
\qquad
\left|\mathsf L_\lambda\right|
\ge\frac{\mathsf I_\lambda}{\left(\mathsf J_*^\lambda\right)^2}.
\]
Since $\mathsf I_\lambda>0$ almost surely and $\mathsf J^\lambda>0$,
we have $\mathsf L_\lambda<0$ almost surely. H\"older's inequality
and the preceding moment bounds imply
\[
\sup_{\substack{x\ge1,\,\theta\in\R\\\left|\lambda\right|\le1}}
\E\left[\left|\mathsf L_\lambda\right|^{-r}\right]<\infty,
\qquad r>0.
\]

We record the positive-order derivative bounds needed below.
All positive-order spatial derivatives of $b_u^\lambda$, and all
derivatives of $c_1,c_2$, are bounded independently of $x,\theta,u$
and $|\lambda|\le1$. The forcing in the equation for
$\mathsf L^\lambda$ is also bounded. The Burkholder--Davis--Gundy
and H\"older inequalities give, for every $r\ge2$,
\[
\E\left[\sup_{-1\le v\le u}
\left|\mathsf L_v^\lambda\right|^r\right]
\le C_{a,r}
+C_{a,r}\int_{-1}^u
\E\left[\sup_{-1\le v\le t}
\left|\mathsf L_v^\lambda\right|^r\right]\dd t.
\]
Gronwall's inequality gives a uniform bound. For
$\mathsf P_u^\lambda=\partial_\lambda^2\mathsf\alpha_u^\lambda$,
the differentiated equation is
\[
\begin{aligned}
\dd\mathsf P_u^\lambda
={}&\left[
\left(b_u^\lambda\right)'\left(\mathsf\alpha_u^\lambda\right)
\mathsf P_u^\lambda
+\left(b_u^\lambda\right)''\left(\mathsf\alpha_u^\lambda\right)
\left(\mathsf L_u^\lambda\right)^2
+2c_1'\left(\mathsf\alpha_u^\lambda\right)\mathsf L_u^\lambda
\right]\dd u
\\
&+\sum_{j=1}^2\left[
c_j'\left(\mathsf\alpha_u^\lambda\right)\mathsf P_u^\lambda
+c_j''\left(\mathsf\alpha_u^\lambda\right)
\left(\mathsf L_u^\lambda\right)^2
\right]\dd\mathsf W_{j,u},
\qquad \mathsf P_{-1}^\lambda=0.
\end{aligned}
\]
Its additional terms have uniformly bounded $r$th moments, using
the first-derivative estimate at order $2r$. At order $k$, the
highest derivative again appears linearly with the same bounded
coefficients. The remaining terms are finite sums of products of
bounded coefficient derivatives and parameter derivatives of orders
less than $k$. Induction, using H\"older's inequality at the required
higher moment orders, followed by Burkholder--Davis--Gundy and
Gronwall, gives
\[
\sup_{\substack{x\ge1,\,\theta\in\R\\\left|\lambda\right|\le1}}
\E\left[
\sup_{-1\le u\le0}
\left|\partial_\lambda^k\mathsf\alpha_u^\lambda\right|^r
\right]<\infty,
\qquad k\ge1,\quad r\ge2.
\]
2These estimates also
hold on every bounded $\lambda$-interval. The fundamental theorem
of calculus and the estimate at the next derivative order show that
the derivatives are continuous in $\lambda$ in the process norm
$(\E[\sup_{-1\le u\le0}|\mathsf X_u|^r])^{1/r}$; the estimate one
order further justifies differentiation in this norm.
In particular, terminal values have the corresponding properties in
$\mathrm{L}^r(\Prob)$, whose norm is
$\|\mathsf X\|_r=(\E[|\mathsf X|^r])^{1/r}$.
The bounds concern positive-order derivatives only: the term
$ax\ee^{au/2}/2$ disappears from the differentiated equations, and
no moment bound on $\mathsf F_\lambda$ uniform in $x,\theta$ is used.

The reciprocal rule gives, for example,
\[
\partial_\lambda\left(\mathsf L_\lambda^{-1}\right)
=-\mathsf L_\lambda^{-2}\partial_\lambda\mathsf L_\lambda.
\]
At any order, the derivative is a finite sum of products of negative
powers of $\mathsf L_\lambda$ and positive-order derivatives of
$\mathsf L_\lambda$. H\"older's inequality bounds all their moments
uniformly. These rules also hold in $\mathrm{L}^r(\Prob)$: first differentiate
$(\mathsf L_\lambda-\delta)^{-1}$, $\delta>0$, and then let
$\delta\downarrow0$. Since $\mathsf L_\lambda<0$, the reciprocal
factors are bounded by powers of $|\mathsf L_\lambda|^{-1}$.
The differences from their limits are bounded by $\delta$ times
products of such powers and parameter derivatives, so convergence
of each derivative holds in $\mathrm{L}^r$, uniformly on bounded
$\lambda$-intervals. Thus
\[
\sup_{\substack{x\ge1,\,\theta\in\R\\\left|\lambda\right|\le1}}
\E\left[
\left|\partial_\lambda^k\left(\mathsf L_\lambda^{-1}\right)\right|^r
\right]<\infty,
\qquad k\ge0,\quad r\ge2,
\]
and these derivatives are continuous in $\lambda$ in every such norm.

\textbf{Step 3. Cameron–Martin translation.} We now use Gaussian integration by parts along a deterministic
translation of the first Brownian path. Put
\[
h\left(u\right)\defeq u+1,\qquad
\tau_\eta\left(w_1,w_2\right)\defeq
\left(w_1+\eta h,w_2\right),\qquad
\mathsf Z\defeq\mathsf W_{1,0}.
\]
The random variable $\mathsf Z$ is standard normal. For a bounded
measurable random variable $\mathsf G:\Omega\to\R$, the
Cameron--Martin formula \cite{CameronMartin1944} gives
\begin{equation}\label{eq:cameron-martin-shift}
\E\left[\mathsf G\circ\tau_\eta\right]
=
\E\left[\mathsf G\ee^{\eta\mathsf Z-\eta^2/2}\right].
\end{equation}
By truncation, the identity also holds whenever its right-hand side
is absolutely integrable; in particular, it holds for
$\mathsf G\in \mathrm{L}^r(\Prob)$ with $r>1$.

The translated Wiener measure is equivalent to the original one.
Approximating It\^o integrals by adapted step sums therefore permits
translation of \eqref{eq:shifted-phase-sde}: replacing
$\mathsf W_1$ by $\mathsf W_1+\eta h$ adds
$\eta$ times the first diffusion coefficient to the drift.
Pathwise uniqueness gives, for each fixed $\lambda,\eta$,
\[
\mathsf\alpha^\lambda\circ\tau_\eta
=\mathsf\alpha^{\lambda+\eta},\qquad
\mathsf F_\lambda\circ\tau_\eta=\mathsf F_{\lambda+\eta},
\quad\text{almost surely}.
\]
The Cameron--Martin formula and H\"older's inequality also imply
\[
\left\|\mathsf X\circ\tau_\eta\right\|_r
\le\ee^{\eta^2/\left(2r\right)}
\left\|\mathsf X\right\|_{2r},\qquad r\ge1.
\]
Thus we may pass to limits of parameter difference quotients under
translation. For every integer $k\ge0$,
\[
\left(\partial_\lambda^k\mathsf F_\lambda\right)\circ\tau_\eta
=\left.\partial_\mu^k\mathsf F_\mu\right|_{\mu=\lambda+\eta}.
\]
The same identity holds for derivatives of $\mathsf L_\lambda^{-1}$.
Moreover, $\mathsf Z\circ\tau_\eta=\mathsf Z+\eta$.

\textbf{Step 4. Integration by parts and Fourier decay.} Define the random weights
\[
\mathsf H_0^\lambda\defeq1,
\qquad
\mathsf H_{j+1}^\lambda\defeq
\left(\mathsf Z+\lambda\right)
\frac{\mathsf H_j^\lambda}{\mathsf L_\lambda}
-\partial_\lambda\left(
\frac{\mathsf H_j^\lambda}{\mathsf L_\lambda}
\right),\qquad j\ge0.
\]
For example,
\[
\mathsf H_1^\lambda
=\frac{\mathsf Z+\lambda}{\mathsf L_\lambda}
+\frac{\partial_\lambda\mathsf L_\lambda}{\mathsf L_\lambda^2}.
\]
Induction using the translation identities gives
$\mathsf H_j^\lambda=\mathsf H_j^0\circ\tau_\lambda$ almost surely
for each fixed $\lambda$. Every parameter derivative of a weight
is a finite sum of products of $\mathsf Z+\lambda$, positive-order
derivatives of $\mathsf F_\lambda$, and negative powers of
$\mathsf L_\lambda$. The preceding bounds and H\"older's inequality
therefore give
\[
\sup_{\substack{x\ge1,\,\theta\in\R\\\left|\lambda\right|\le1}}
\E\left[
\left|\partial_\lambda^k\mathsf H_j^\lambda\right|^r
\right]<\infty,
\qquad j,k\ge0,\quad r\ge2.
\]
These derivatives are also continuous in $\mathrm{L}^r$.

Let $f:\R\to\R$ be smooth and bounded, with bounded derivatives.
For fixed $j$, apply \eqref{eq:cameron-martin-shift} with
$\eta=\lambda$ and
\[
\mathsf G\defeq
f\left(\mathsf F_0\right)\frac{\mathsf H_j^0}{\mathsf L_0}.
\]
The shift identities established above give
\[
\mathsf G\circ\tau_\lambda
=
f\left(\mathsf F_\lambda\right)
\frac{\mathsf H_j^\lambda}{\mathsf L_\lambda}.
\]
Consequently,
\begin{equation}\label{eq:cameron-martin-weighted}
\E\left[
f\left(\mathsf F_\lambda\right)
\frac{\mathsf H_j^\lambda}{\mathsf L_\lambda}
\right]
=
\E\left[
f\left(\mathsf F_0\right)
\frac{\mathsf H_j^0}{\mathsf L_0}
\ee^{\lambda\mathsf Z-\lambda^2/2}
\right].
\end{equation}
The integrand on the left is differentiable in
$\mathrm L^1(\Prob)$ by the preceding estimates.
On the right, $\mathsf G$ has moments of every order, and
\[
\left|
\partial_\lambda
\ee^{\lambda\mathsf Z-\lambda^2/2}
\right|
=
\left|\mathsf Z-\lambda\right|
\ee^{\lambda\mathsf Z-\lambda^2/2}
\le
\left(\left|\mathsf Z\right|+1\right)\ee^{\left|\mathsf Z\right|},
\qquad |\lambda|\le1.
\]
H\"older's inequality and Gaussian exponential moments show that
the last expression multiplied by $|\mathsf G|$ is integrable.
We may therefore differentiate
\eqref{eq:cameron-martin-weighted} at $\lambda=0$.
Since $\partial_\lambda\mathsf F_\lambda=\mathsf L_\lambda$,
the product rule gives
\[
\begin{aligned}
\E\left[f'\left(\mathsf F_0\right)\mathsf H_j^0\right]
&=\E\left[
f\left(\mathsf F_0\right)
\left\{
\mathsf Z\frac{\mathsf H_j^0}{\mathsf L_0}
-\left.\partial_\lambda\left(
\frac{\mathsf H_j^\lambda}{\mathsf L_\lambda}
\right)\right|_{\lambda=0}
\right\}\right]
=\E\left[f\left(\mathsf F_0\right)\mathsf H_{j+1}^0\right].
\end{aligned}
\]
Starting from $\mathsf H_0^0=1$ and iterating yields
\[
\E\left[f^{(m)}\left(\mathsf F_0\right)\right]
=\E\left[f\left(\mathsf F_0\right)\mathsf H_m^0\right].
\]
Applying this to the real and imaginary parts of $f(y)=\ee^{\ii ny}$,
we obtain
\[
\mathfrak K_{x,n}\left(\theta;a\right)
=\left(\ii n\right)^{-m}
\E\left[\ee^{\ii n\mathsf F_0}\mathsf H_m^0\right],
\qquad n\ne0.
\]
The uniform first moment of $\mathsf H_m^0$ gives
$|\mathfrak K_{x,n}(\theta;a)|\le C_{a,m}|n|^{-m}$.
Conditioning as at the beginning proves the same bound for
$\mathfrak q_n(x;a)$. Since $m\ge s$ and $\mathfrak q_0(x;a)=1$,
the lemma follows.
\end{proof}

\begin{rmk}
Lemma~\ref{lem:uniform-modes} implies that, for fixed $a>0$
and every $x\ge1$, the law of $\mathsf\alpha_0(x;a)$ modulo
$2\pi$ has a smooth density $\varrho(x,\phi;a)$ relative to
$\dd\phi/(2\pi)$ on $\mathbb T=\R/(2\pi\Z)$ given by,
\[
\varrho\left(x,\phi;a\right)
=\sum_{n\in\Z}\mathfrak q_n\left(x;a\right)\ee^{-\ii n\phi}.
\]
The conclusion of the lemma is, by standard arguments,  equivalent to
\[
\sup_{x\ge1}
\left\|\varrho\left(x,\cdot;a\right)\right\|_{\mathrm C^k(\mathbb T)}
<\infty,
\qquad k=0,1,\ldots.
\]
We also observe that the moment equations~\eqref{eq:q-hierarchy} take the form
\[
\partial_x\varrho
=-\partial_\phi\varrho
+\frac1x\left[
2\partial_\phi\left(\sin\phi\,\varrho\right)
+\frac2a\partial_\phi^2\left((1-\cos\phi)\varrho\right)
\right],
\qquad x>1.
\]
 One could alternatively try to prove
Lemma~\ref{lem:uniform-modes} using a quantitative version of
H\"ormander's theorem with estimates uniform in $x\ge1$.
This is natural, since the argument above is a Malliavin
calculus argument in disguise. The correlation function $\mathcal{R}_a(x)$ also also has a phase-space representation in terms of $\varrho (x,\phi;a)$ but will not make use of it.
\end{rmk}

\subsection{Corrected Fourier modes}

Throughout this subsection, $a>0$ is real and $p=2/a$.
Write $\boldsymbol{\mathfrak q}(x;a)=(\mathfrak q_n(x;a))_{n\in\Z}$,
and let $\mathcal N$ denote multiplication by the index:
$(\mathcal N\mathbf u)_n=nu_n$.
Define the matrix $\mathbf B(a)$ by
\[
\begin{gathered}
\mathbf B_{n,n}\left(a\right)\defeq-pn^2,
\qquad
\mathbf B_{n,n-1}\left(a\right)\defeq\frac{n\left(n+a\right)}a,
\qquad
\mathbf B_{n,n+1}\left(a\right)\defeq\frac{n\left(n-a\right)}a,
\qquad n\in\Z,
\end{gathered}
\]
with all other entries zero. Thus \eqref{eq:q-hierarchy} reads
\[
\partial_x\boldsymbol{\mathfrak q}
=\left(\ii\mathcal N+\frac{\mathbf B\left(a\right)}x\right)
\boldsymbol{\mathfrak q}.
\]
To remove the interaction of order $x^{-1}$ between neighbouring
modes, define
\begin{equation}\label{eq:normalform}
\mathbf T_{n,n-1}\left(a\right)\defeq
-\ii\frac{n\left(n+a\right)}a,
\qquad
\mathbf T_{n,n+1}\left(a\right)\defeq
\ii\frac{n\left(n-a\right)}a,
\qquad n\in\Z,
\end{equation}
with all other entries zero. These choices give the cancellation
\[
\mathbf B\left(a\right)
+\ii\left(\mathbf T\left(a\right)\mathcal N
-\mathcal N\mathbf T\left(a\right)\right)
=-p\mathcal N^2.
\]
We also put
\begin{equation}\label{eq:finite-band}
\mathbf F\left(a\right)\defeq
\mathbf T\left(a\right)\mathbf B\left(a\right)
-\mathbf T\left(a\right)+p\mathcal N^2\mathbf T\left(a\right).
\end{equation}
The change of variables using $\mathbf T(a)$ in the proof below
is of Harris--Lutz type \cite{HarrisLutz} and appears to be rather standard in the literature on asymptotic analysis of ordinary differential equations.
The rows of $\mathbf T(a)$ and $\mathbf F(a)$ involve, respectively,
only the indices $n-1,n+1$ and $n-2,\ldots,n+2$, with coefficients
growing at most quadratically and quartically in $n$.
Their zeroth rows vanish.

The following proposition gives the correlation asymptotics, leaving
only the oscillatory coefficient to be identified.
\begin{prop}\label{prop:real-asymptotics}
For real $a>1$, the limit
\begin{equation}\label{eq:amplitude-definition}
\mathfrak U\left(a\right)\defeq
t_1\left(a\right)\lim_{x\to\infty}
\ee^{-\ii x}x^{2/a}
\left(
\mathfrak q_1\left(x;a\right)-\frac{\ii\left(a+1\right)}{ax}
\right)
\end{equation}
exists. For each fixed real $a>0$ with $a\ne1$, as $x\to\infty$,
\begin{equation}\label{eq:real-asymptotics}
\mathcal R_a\left(x\right)=
\begin{cases}
1-\dfrac{2}{ax^2}+o\left(x^{-2}\right),&0<a<1,\\[4pt]
1+2\Re\left(\mathfrak U\left(a\right)\ee^{\ii x}\right)x^{-2/a}
+o\left(x^{-2/a}\right),&a>1.
\end{cases}
\end{equation}
\end{prop}

\begin{proof}
The first step is to separate the algebraic terms forced by the
constant mode $\mathfrak q_0=1$ from the remaining behaviour of the
nonzero modes. We construct a finite expansion in inverse powers
of $x$ and subtract it from the moments, choosing its coefficients
so that the residual equation has a forcing term of higher order.
This removes the explicit algebraic background before estimating
the residual. We then correct the residual to cancel the leading
interaction between neighbouring modes. The resulting equation
has a diagonal part that displays both the oscillation and the
damping of each mode, while the remaining interactions carry an
additional power of $x^{-1}$. These equations allow us to isolate
the first oscillatory mode when $p<2$, and to show that the residual
is negligible at order $x^{-2}$ when $p>2$.

\textbf{Step 1. Common  analysis of both regimes.}
We first construct a finite algebraic approximation to the moments.
Its coefficients are defined directly from the hierarchy, without
assuming that an asymptotic expansion exists.
Let $\mathbf e_j$, $j\in\Z$, be the unit vector at index $j$, and define
$\mathcal N^{-1}\mathbf e_0=0$ and
$\mathcal N^{-1}\mathbf e_n=n^{-1}\mathbf e_n$ for $n\ne0$.
With $\mathbf I$ denoting the identity, put
\begin{equation}\label{eq:formal-mode-coefficients}
\begin{gathered}
\mathbf s_0\left(a\right)\defeq\mathbf e_0,
\qquad
\mathbf s_{j+1}\left(a\right)\defeq
\ii\mathcal N^{-1}\left(\mathbf B\left(a\right)+j\mathbf I\right)
\mathbf s_j\left(a\right),\qquad j=0,1,2,
\\
\mathbf Q\left(x;a\right)\defeq
\mathbf e_0+\frac{\mathbf s_1\left(a\right)}x
+\frac{\mathbf s_2\left(a\right)}{x^2}
+\frac{\mathbf s_3\left(a\right)}{x^3}.
\end{gathered}
\end{equation}
Each $\mathbf s_j$ is supported on $\{-j,\ldots,j\}$, and
its zeroth coordinate vanishes for $j\ge1$ by the definition of
$\mathcal N^{-1}$. We shall compute only
$\mathbf s_1$ and $\mathbf s_2$. The auxiliary term
$\mathbf s_3/x^3$ makes the remaining explicit forcing of
order $x^{-4}$.
The first coefficient is
\begin{equation}\label{eq:first-formal-mode}
\mathbf s_1\left(a\right)
=-\mathbf T\left(a\right)\mathbf e_0
=\ii\frac{a+1}{a}\left(\mathbf e_1-\mathbf e_{-1}\right).
\end{equation}

Set
\[
\boldsymbol{\mathfrak r}\left(x;a\right)\defeq
\boldsymbol{\mathfrak q}\left(x;a\right)-\mathbf Q\left(x;a\right),
\qquad
\mathbf g\left(a\right)\defeq
\left(\mathbf B\left(a\right)+3\mathbf I\right)\mathbf s_3\left(a\right).
\]
Since the zeroth row of $\mathbf B(a)$ vanishes, the recurrence
in \eqref{eq:formal-mode-coefficients} gives
$\ii\mathcal N\mathbf s_{j+1}=-(\mathbf B+j\mathbf I)\mathbf s_j$.
Substitution into \eqref{eq:q-hierarchy} therefore yields the
exact equation
\[
\partial_x\boldsymbol{\mathfrak r}
=\left(\ii\mathcal N+\frac{\mathbf B\left(a\right)}x\right)
\boldsymbol{\mathfrak r}+\frac{\mathbf g\left(a\right)}{x^4}.
\]
To move the remaining neighbouring interactions to order $x^{-2}$,
define the corrected residual using \eqref{eq:normalform}.
The cancellation above and direct differentiation give
\begin{equation}\label{eq:corrected-residual}
\begin{gathered}
\boldsymbol{\mathfrak w}\left(x;a\right)\defeq
\boldsymbol{\mathfrak r}\left(x;a\right)
+\frac{\mathbf T\left(a\right)\boldsymbol{\mathfrak r}\left(x;a\right)}x,
\\
\partial_x\boldsymbol{\mathfrak w}
=\left(\ii\mathcal N-\frac{p\mathcal N^2}x\right)
\boldsymbol{\mathfrak w}
+\frac{\mathbf F\left(a\right)\boldsymbol{\mathfrak r}}{x^2}
+\frac{\mathbf g\left(a\right)}{x^4}
+\frac{\mathbf T\left(a\right)\mathbf g\left(a\right)}{x^5}.
\end{gathered}
\end{equation}
Here $\mathfrak r_0=\mathfrak w_0=0$, and the two explicit
source vectors have finite support. All these identities are
coordinatewise finite sums. To recover $\boldsymbol{\mathfrak r}$
we shall use the identity
$\boldsymbol{\mathfrak r}
=\boldsymbol{\mathfrak w}-x^{-1}\mathbf T(a)\boldsymbol{\mathfrak r}$
together with bounds already obtained for
$\boldsymbol{\mathfrak r}$.

By Lemma~\ref{lem:uniform-modes} and the finite support of
the coefficients of $\mathbf Q$, for every $s\ge0$,
\[
\sup_{x\ge1}\sup_{n\ne0}
\left(1+\left|n\right|\right)^s
\left|\mathfrak r_n\left(x;a\right)\right|<\infty.
\]
The same bound holds for $\mathbf T(a)\boldsymbol{\mathfrak r}$
and $\mathbf F(a)\boldsymbol{\mathfrak r}$, and for the initial
values $\mathfrak w_n(1;a)$. Indeed, each row of $\mathbf T(a)$
and $\mathbf F(a)$ contains finitely many neighbouring
coordinates, with coefficients of degree at most two and four
in $n$, respectively. A bound with weight $s+2$ or $s+4$ for
$\boldsymbol{\mathfrak r}$ therefore gives a bound with weight
$s$ after applying these matrices. The same observation applies
to every weighted decay estimate below.

The scalar solution factor for coordinate $n\ne0$ in
\eqref{eq:corrected-residual} is
$\ee^{\ii n(x-t)}(t/x)^{pn^2}$. We shall use
\begin{equation}\label{eq:scalar-damping}
\int_1^x\left(\frac tx\right)^{pn^2}t^{-\rho-1}\,\dd t
=\frac{x^{-\rho}-x^{-pn^2}}{pn^2-\rho}
\le\frac{x^{-\rho}}{pn^2-\rho},
\qquad 0<\rho<pn^2.
\end{equation}
Put $\mu=\min(1,p/2)$, so that $\mu\le1$ and $\mu<p$.
The preceding uniform bounds and variation of constants give
\[
\begin{aligned}
\left(1+\left|n\right|\right)^s
\left|\mathfrak w_n\left(x;a\right)\right|
&\le C_{a,s}\left[
x^{-pn^2}+\int_1^x\left(\frac tx\right)^{pn^2}t^{-2}\,\dd t
\right]
\le C_{a,s}\left[
x^{-p}+\frac{x^{-\mu}}{pn^2-\mu}
\right]
\le C_{a,s}x^{-\mu}.
\end{aligned}
\]
We used $t^{-2}\le t^{-\mu-1}$ and
$pn^2-\mu\ge p-\mu>0$. The uniform weighted bound on
$\mathbf T(a)\boldsymbol{\mathfrak r}$ shows that replacing
$\mathfrak w_n$ by $\mathfrak r_n$ adds at most
$C_{a,s}x^{-1}$. Consequently,
\begin{equation}\label{eq:preliminary-residual-bound}
\sup_{n\ne0}\left(1+\left|n\right|\right)^s
\left|\mathfrak r_n\left(x;a\right)\right|
\le C_{a,s}x^{-\mu},
\qquad s\ge0,\quad x\ge1.
\end{equation}

\textbf{Step 2. The $a>1$ regime.} Suppose first that $a>1$, so $0<p<2$ and $\mu=p/2$.
The first row of \eqref{eq:corrected-residual} gives
\[
\begin{aligned}
\frac{\dd}{\dd x}
\left(\ee^{-\ii x}x^p\mathfrak w_1\left(x;a\right)\right)
={}&\ee^{-\ii x}x^{p-2}
\left[\mathbf F\left(a\right)
\boldsymbol{\mathfrak r}\left(x;a\right)\right]_1
+\ee^{-\ii x}x^{p-4}\mathbf g_1\left(a\right)
+\ee^{-\ii x}x^{p-5}
\left[\mathbf T\left(a\right)\mathbf g\left(a\right)\right]_1.
\end{aligned}
\]
By \eqref{eq:preliminary-residual-bound}, the first term is
$O(x^{p-2-\mu})$. It is absolutely integrable because
$p-2-\mu=p/2-2<-1$. The other two terms are also absolutely
integrable, since $p<2$. Thus
$\ee^{-\ii x}x^p\mathfrak w_1(x;a)$ has a finite limit.

To recover the first moment, use
\[
\boldsymbol{\mathfrak q}
=\mathbf Q+\boldsymbol{\mathfrak w}
-\frac{\mathbf T\left(a\right)\boldsymbol{\mathfrak r}}x.
\]
By \eqref{eq:first-formal-mode},
$(\mathbf Q(x;a))_1=\ii(a+1)/(ax)+O(x^{-2})$.
Also, since $\mathfrak r_0=0$, the finite bandwidth and
quadratic coefficient bounds for $\mathbf T(a)$ give, for
every $s\ge0$,
\[
\begin{aligned}
\sup_{n\ne0}(1+|n|)^s
\left|
\frac{[\mathbf T(a)\boldsymbol{\mathfrak r}(x;a)]_n}{x}
\right|
&\le C_{a,s}x^{-1}
\sup_{m\ne0}(1+|m|)^{s+2}|\mathfrak r_m(x;a)|\le C_{a,s}x^{-1-\mu},
\ \ x\ge1,
\end{aligned}
\]
where the second inequality follows from
\eqref{eq:preliminary-residual-bound} with weight $s+2$. Hence
\[
\mathfrak q_1\left(x;a\right)-\frac{\ii\left(a+1\right)}{ax}
=\mathfrak w_1\left(x;a\right)
+O\left(x^{-2}\right)+O\left(x^{-1-\mu}\right).
\]
Since $2>p$ and $1+\mu>p$, the two error terms are $o(x^{-p})$.
Thus the limit in \eqref{eq:amplitude-definition} exists and equals
$t_1(a)\lim_{x\to\infty}\ee^{-\ii x}x^p\mathfrak w_1(x;a)$.
Consequently,
\[
\mathfrak q_1\left(x;a\right)
=\frac{\ii\left(a+1\right)}{ax}
+\frac{\mathfrak U\left(a\right)}{t_1\left(a\right)}
\ee^{\ii x}x^{-p}
+o\left(x^{-p}\right).
\]
The explicit term of order $x^{-1}$ is purely imaginary.

For $|n|\ge2$, the forcing in \eqref{eq:corrected-residual}
is uniformly $O(x^{-2-\mu})$. Since $pn^2\ge4p$ and
$0<t/x\le1$, variation of constants gives
\[
\sup_{\left|n\right|\ge2}
\left|\mathfrak w_n\left(x;a\right)\right|
\le C_a\left[
x^{-4p}+x^{-4p}\int_1^x t^{4p-2-\mu}\,\dd t
\right]
=o\left(x^{-p}\right).
\]
Indeed, the integral contributes $O(x^{-1-\mu})$ if
$4p>1+\mu$, $O(x^{-4p})$ if $4p<1+\mu$, and
$x^{-4p}\log x$ when $4p=1+\mu$. Each is $o(x^{-p})$
because $4p>p$ and $1+\mu>p$.
On these coordinates, $\mathbf Q=O(x^{-2})$, since
$\mathbf s_1$ is supported on $\{-1,1\}$.
The correction $x^{-1}\mathbf T(a)\boldsymbol{\mathfrak r}$
is uniformly $O(x^{-1-\mu})$, so
\[
\sup_{\left|n\right|\ge2}
\left|\mathfrak q_n\left(x;a\right)\right|
=o\left(x^{-p}\right).
\]
The weights $t_n(a)$ are real and absolutely summable. Hence
\[
\left|
\sum_{n=2}^{\infty}t_n\left(a\right)
\Re\mathfrak q_n\left(x;a\right)
\right|
\le
\left(\sum_{n=2}^{\infty}\left|t_n\left(a\right)\right|\right)
\sup_{n\ge2}\left|\mathfrak q_n\left(x;a\right)\right|
=o\left(x^{-p}\right).
\]
Taking real parts of the first-mode expansion and substituting
into \eqref{eq:correlation-mode-sum} proves
\[
\mathcal R_a\left(x\right)
=1+2\Re\left(\mathfrak U\left(a\right)\ee^{\ii x}\right)x^{-p}
+o\left(x^{-p}\right).
\]

\textbf{Step 3. The $a<1$ regime.} Now suppose that $0<a<1$, so $p>2$ and $\mu=1$.
The preliminary bound~\eqref{eq:preliminary-residual-bound},
with $\mu=1$, gives
\[
\sup_{n\ne0}(1+|n|)^s|\mathfrak r_n(x;a)|
\le C_{a,s}x^{-1},
\qquad s\ge0,\quad x\ge1.
\]
Throughout the following estimates, constants may depend on
the fixed parameter $a$ and the indicated weight, but are
independent of $x\ge1$ and the mode index.

Since $\mathfrak r_0=0$, the finite bandwidth and coefficient
bounds for $\mathbf T(a)$ and $\mathbf F(a)$ imply
\[
\begin{aligned}
\sup_{n\ne0}(1+|n|)^s
\left|[\mathbf T(a)\boldsymbol{\mathfrak r}(x;a)]_n\right|
&\le C_{a,s}
\sup_{m\ne0}(1+|m|)^{s+2}|\mathfrak r_m(x;a)|,\\
\sup_{n\ne0}(1+|n|)^s
\left|[\mathbf F(a)\boldsymbol{\mathfrak r}(x;a)]_n\right|
&\le C_{a,s}
\sup_{m\ne0}(1+|m|)^{s+4}|\mathfrak r_m(x;a)|.
\end{aligned}
\]
Indeed, each row involves at most five indices $m$ with
$|m-n|\le2$, for which $1+|n|\le3(1+|m|)$, and the
coefficients grow at most quadratically and quartically,
respectively. Thus the constants in these estimates are
independent of the row index.

The vectors $\mathbf g(a)$ and $\mathbf T(a)\mathbf g(a)$
have finite support. Using the preliminary bound with
weight $s+4$, the forcing in~\eqref{eq:corrected-residual}
therefore satisfies
\[
\begin{aligned}
&\sup_{n\ne0}(1+|n|)^s
\left|
\frac{[\mathbf F(a)\boldsymbol{\mathfrak r}(x;a)]_n}{x^2}
+\frac{\mathbf g_n(a)}{x^4}
+\frac{[\mathbf T(a)\mathbf g(a)]_n}{x^5}
\right|\\
&\qquad\le C_{a,s}\left[
x^{-2}\sup_{m\ne0}(1+|m|)^{s+4}
|\mathfrak r_m(x;a)|+x^{-4}+x^{-5}
\right]
\le C_{a,s}x^{-3}.
\end{aligned}
\]
The preliminary bound and the estimate for $\mathbf T(a)$
also give
\[
\sup_{n\ne0}(1+|n|)^s|\mathfrak w_n(1;a)|
\le C_{a,s}.
\]
Variation of constants and~\eqref{eq:scalar-damping} with
$\rho=2$ consequently yield, for every $n\ne0$,
\[
\begin{aligned}
(1+|n|)^s|\mathfrak w_n(x;a)|
&\le C_{a,s}\left[
x^{-pn^2}
+\int_1^x\left(\frac tx\right)^{pn^2}t^{-3}\,\dd t
\right]\\
&\le C_{a,s}\left[
x^{-p}+\frac{x^{-2}}{pn^2-2}
\right]\le C_{a,s}x^{-2},
\qquad x\ge1.
\end{aligned}
\]
To transfer this estimate to $\boldsymbol{\mathfrak r}$,
use the preliminary bound with weight $s+2$:
\[
\begin{aligned}
\sup_{n\ne0}(1+|n|)^s
\left|
\frac{[\mathbf T(a)\boldsymbol{\mathfrak r}(x;a)]_n}{x}
\right|
&\le C_{a,s}x^{-1}
\sup_{m\ne0}(1+|m|)^{s+2}|\mathfrak r_m(x;a)|\le C_{a,s}x^{-2}.
\end{aligned}
\]
The identity
$\boldsymbol{\mathfrak r}
=\boldsymbol{\mathfrak w}
-x^{-1}\mathbf T(a)\boldsymbol{\mathfrak r}$
therefore gives
\begin{equation}\label{eq:second-order-residual-bound}
\sup_{n\ne0}(1+|n|)^s|\mathfrak r_n(x;a)|
\le C_{a,s}x^{-2},
\qquad s\ge0,\quad x\ge1.
\end{equation}
We now use this improved estimate with weight $4$.
The forcing in~\eqref{eq:corrected-residual} satisfies
\[
\begin{aligned}
&\sup_{n\ne0}
\left|
\frac{[\mathbf F(a)\boldsymbol{\mathfrak r}(x;a)]_n}{x^2}
+\frac{\mathbf g_n(a)}{x^4}
+\frac{[\mathbf T(a)\mathbf g(a)]_n}{x^5}
\right|\\
&\qquad\le C_a\left[
x^{-2}\sup_{m\ne0}(1+|m|)^4|\mathfrak r_m(x;a)|
+x^{-4}+x^{-5}
\right]
\le C_ax^{-4}.
\end{aligned}
\]
Thus variation of constants gives
\[
\begin{aligned}
\sup_{n\ne0}|\mathfrak w_n(x;a)|
\le C_a\left[
x^{-p}+\int_1^x\left(\frac tx\right)^p t^{-4}\,\dd t
\right]\le C_a
\begin{cases}
x^{-p},&2<p<3,\\
x^{-3}(1+\log x),&p=3,\\
x^{-3},&p>3.
\end{cases}
\end{aligned}
\]
Every case is $o(x^{-2})$. Moreover,
\eqref{eq:second-order-residual-bound} with weight $2$
gives
\[
\begin{aligned}
\sup_{n\ne0}
\left|
\frac{[\mathbf T(a)\boldsymbol{\mathfrak r}(x;a)]_n}{x}
\right|
&\le C_ax^{-1}
\sup_{m\ne0}(1+|m|)^2|\mathfrak r_m(x;a)|\le C_ax^{-3}.
\end{aligned}
\]
Consequently,
\[
\sup_{n\ne0}|\mathfrak r_n(x;a)|
\le
\sup_{n\ne0}|\mathfrak w_n(x;a)|+C_ax^{-3}
=o(x^{-2}).
\]
Finally, $\mathbf s_3(a)$ has finite support. Since
$\boldsymbol{\mathfrak q}=\mathbf Q+\boldsymbol{\mathfrak r}$,
we obtain
\[
\begin{aligned}
\sup_{n\ne0}
\left|
\mathfrak q_n(x;a)
-\frac{(\mathbf s_1(a))_n}{x}
-\frac{(\mathbf s_2(a))_n}{x^2}
\right|
&\le
\sup_{n\ne0}|\mathfrak r_n(x;a)|+C_ax^{-3}=o(x^{-2}).
\end{aligned}
\]
This is the required second-order expansion, with a remainder
uniform in the mode index.

The coefficient $\mathbf s_1$ was computed in
\eqref{eq:first-formal-mode} and is purely imaginary.
For $\mathbf s_2=\ii\mathcal N^{-1}(\mathbf B+\mathbf I)\mathbf s_1$,
the positive coordinates are
\[
\left(\mathbf s_2\left(a\right)\right)_n
=
\begin{cases}
\left(p-1\right)\dfrac{a+1}{a},&n=1,\\[4pt]
-\dfrac{\left(a+1\right)\left(a+2\right)}{a^2},&n=2,\\[4pt]
0,&n\ge3.
\end{cases}
\]
Indeed, the first coordinate comes from the diagonal coefficient
$1-p$ acting on $(\mathbf s_1)_1$, while the second comes
from $\mathbf B_{2,1}=2(a+2)/a$ followed by multiplication
by $\ii/2$. No higher positive coordinate occurs.
Since $\sum_{n\ge1}|t_n(a)|<\infty$, the uniform
$o(x^{-2})$ remainder remains $o(x^{-2})$ after summation
in \eqref{eq:correlation-mode-sum}. Using
\[
t_1\left(a\right)=-\frac a{a+1},
\qquad
t_2\left(a\right)=
\frac{a\left(a-1\right)}{\left(a+1\right)\left(a+2\right)},
\]
the coefficient of $x^{-2}$ is
\[
2\left[
t_1\left(a\right)\left(p-1\right)\frac{a+1}{a}
-t_2\left(a\right)\frac{\left(a+1\right)\left(a+2\right)}{a^2}
\right]
=2\left[-\left(p-1\right)-\frac{a-1}{a}\right]
=-p,
\]
where $p=2/a$. Hence
$\mathcal R_a(x)=1-px^{-2}+o(x^{-2})$, proving the first
line of \eqref{eq:real-asymptotics}.
\end{proof}
\section{Identification of the oscillatory coefficient}\label{sec:connection}

We use the Fourier coefficients
\[
 \psi_n\left(a\right)\defeq
 \frac{\left(-1\right)^n\Gamma\left(2a+1\right)}
 {2^{2a}\Gamma\left(a+n+1\right)\Gamma\left(a-n+1\right)},
 \qquad n\ge0.
\]
Reciprocal Gamma factors are understood through their entire extensions.
The next lemma supplies the holomorphic continuation and growth bound
needed to determine the amplitude from its values at integer parameters.
\begin{lem}\label{lem:amplitude-continuation}
The function $\psi_0(a)\mathfrak U(a)$, initially defined for real
$a>1$ by Proposition~\ref{prop:real-asymptotics}, extends holomorphically
to $\Omega=\{a\in\C:\Re a>1\}$. For some constants $C>0$ and
$M\in\N$, its extension satisfies
\begin{equation}\label{eq:amplitude-growth}
 \left|\psi_0\left(a\right)\mathfrak U\left(a\right)\right|
 \le C\left(1+\left|a\right|\right)^M,
 \qquad \Re a>3.
\end{equation}
\end{lem}

\begin{proof}
Put $p=2/a$ and $\sigma=\Re(1/a)\in(0,1)$ for $a\in\Omega$.
We will construct holomorphic functions $\mathcal A_R$ on $\Omega$
which are locally uniformly bounded in $R$, converge to
$\psi_0(a)\mathfrak U(a)$ for every real $a>1$, and satisfy
$|\mathcal A_R(a)|\le C(1+|a|)^5$ on $\Re a>3$, with $C$
independent of $R$. Vitali's theorem will then give the required
holomorphic extension, and the polynomial bound will pass to the
limit. The functions $\mathcal A_R$ will be obtained as large-$x$
limits of finite systems. Uniformity of these limits in $R$ will
allow us to identify their limit with the physical amplitude.

\textbf{Step 1. Finite systems and the energy estimate.}
We first construct finite solutions which are holomorphic in $a$
and satisfy estimates independent of the number of coordinates.
We work with the positive modes, whose equations form a closed
system once the fixed zeroth mode $\mathfrak q_0=1$ is treated
as an explicit forcing. These modes suffice because the amplitude
$\mathfrak U(a)$ is defined through the first moment. Set
$\boldsymbol\psi(a)=(\psi_n(a))_{n\ge1}$ for $a\in\Omega$ and,
for real $a>1$, set
$\mathbf z(x;a)=(\psi_n(a)\mathfrak q_n(x;a))_{n\ge1}$.
The limits $\mathfrak q_n(x;a)\to1$ as $x\downarrow0$ give
$z_n(x;a)\to\psi_n(a)$.
We therefore call $\boldsymbol\psi(a)$ the entrance vector:
it specifies the limiting values at the singular endpoint $x=0$.
For complex $a\in\Omega$, the same explicit vector will supply
the limiting values for the finite systems constructed below.
We use the following norm and inner product conventions:
\[
\begin{gathered}
\left\|\mathbf u\right\|_{\ell^2\left(\N\right)}^2
\defeq\sum_{n=1}^{\infty}\left|u_n\right|^2,
\qquad
\left\|\mathbf v\right\|_{\C^R}^2
\defeq\sum_{n=1}^{R}\left|v_n\right|^2,
\\
\left(\mathbf u,\mathbf v\right)_{\C^R}
\defeq\sum_{n=1}^{R}\overline{u_n}v_n,
\qquad
\left\|\mathbf A\right\|_{\C^R\to\C^R}
\defeq
\sup_{\substack{\mathbf u\in\C^R\\
                \left\|\mathbf u\right\|_{\C^R}=1}}
\left\|\mathbf A\mathbf u\right\|_{\C^R}.
\end{gathered}
\]

The Gamma recurrence and \eqref{eq:q-hierarchy} give, coordinatewise,
\begin{equation}\label{eq:weighted-hierarchy}
\begin{gathered}
\mathbf z'
=\left(\ii\mathcal N+\frac{\mathbf H_a}{x}\right)\mathbf z
+\frac{\mathbf b}{x},
\qquad
\mathbf b\left(a\right)\defeq-\psi_0\left(a\right)\mathbf e_1,
\\
\left(\mathbf H_a\mathbf u\right)_n
=-\frac{2n^2}{a}u_n
+\frac{n\left(n-a-1\right)}a u_{n-1}
+\frac{n\left(n+a+1\right)}a u_{n+1},
\qquad n\ge1,
\end{gathered}
\end{equation}
where $u_0=0$ and $\mathbf e_1$ is the first coordinate vector.
To control the entrance vector, write
\[
f_a\left(\theta\right)\defeq
\left|\sin\left(\frac\theta2\right)\right|^{2a},
\qquad
\psi_n\left(a\right)
=\frac1{2\pi}\int_0^{2\pi}
f_a\left(\theta\right)\ee^{-\ii n\theta}\,\dd\theta,
\qquad n\ge0.
\]
The power uses the real logarithm of $|\sin(\theta/2)|$.
These Fourier integrals are holomorphic in $a\in\Omega$.
Moreover, $|f_a|\le1$, and on $0<\theta<2\pi$,
\[
f_a'\left(\theta\right)
=a\cos\left(\frac\theta2\right)
\sin\left(\frac\theta2\right)^{2a-1}.
\]
Thus $|f_a'|\le|a|$ on $\Omega$.
Parseval's identity, applied to $f_a$ and its periodic derivative,
gives
\begin{equation}\label{eq:entrance-fourier-bounds}
\begin{gathered}
\left\|\boldsymbol\psi\left(a\right)\right\|_{\ell^2\left(\N\right)}^2
\le\frac1{2\pi}\int_0^{2\pi}
\left|f_a\left(\theta\right)\right|^2\,\dd\theta
\le1, \ \ 
\left\|\mathcal N\boldsymbol\psi\left(a\right)\right\|_{\ell^2\left(\N\right)}^2
\le\frac1{2\pi}\int_0^{2\pi}
\left|f_a'\left(\theta\right)\right|^2\,\dd\theta
\le\left|a\right|^2.
\end{gathered}
\end{equation}
The same Fourier integral gives $|\psi_0(a)|\le1$.

Let $\mathbf P_R$ be the projection onto $1\le n\le R$,
where $R\ge3$, and identify its range with $\C^R$.
Write $\mathbf I$ for the identity on $\C^R$, and use
$\mathcal N$ also for its restriction
$\operatorname{diag}(1,\ldots,R)$.
In finite systems, $\mathbf b$ denotes its restriction to these
coordinates.

The entrance vector satisfies
$\mathbf H_a\boldsymbol\psi+\mathbf b=0$ coordinatewise.
We choose the finite system so that $\mathbf P_R\boldsymbol\psi$
satisfies the same balance equation.
An ordinary truncation drops the last-row term
$R(R+a+1)\psi_{R+1}/a$.
The recurrence
$\psi_{R+1}=(R-a)\psi_R/(R+a+1)$ allows us to replace this
missing term by $R(R-a)\psi_R/a$.
Accordingly, let $\mathbf H_{a,R}$ be the $R\times R$
truncation of $\mathbf H_a$ with its last diagonal entry
replaced by
\begin{equation}\label{eq:preserving-truncation}
\left(\mathbf H_{a,R}\right)_{R,R}
\defeq-\frac{R^2}{a}-R.
\end{equation}
This modification gives
\begin{equation}\label{eq:finite-entrance-balance}
\mathbf H_{a,R}\mathbf P_R\boldsymbol\psi+\mathbf b=0.
\end{equation}

We claim the energy estimate
\begin{equation}\label{eq:nonzero-energy}
\Re\left(\mathbf u,\mathbf H_{a,R}\mathbf u\right)_{\C^R}
\le
\kappa\left\|\mathbf u\right\|_{\C^R}^2,
\qquad
\kappa\defeq1-\frac{33}{16}\sigma<1.
\end{equation}
The inequalities $\kappa<1$ and $\kappa+2\sigma<1$ will
control, respectively, the entrance integral at zero and the
weighted integral at infinity. To prove the estimate, put
\[
\mathcal D_R\left(\mathbf u\right)\defeq
\sum_{n=1}^{R-1}
n\left(n+1\right)\left|u_n-u_{n+1}\right|^2
+
2R\left|u_R\right|^2.
\]
Telescoping and Cauchy--Schwarz give
$|u_1|^2\le\mathcal D_R(\mathbf u)$.
Telescoping $\|\mathbf u\|_{\C^R}^2$ and applying Young's
inequality also gives
\[
\begin{aligned}
\left\|\mathbf u\right\|_{\C^R}^2
&\le
\frac12\left\|\mathbf u\right\|_{\C^R}^2
+
2\sum_{n=1}^{R-1}
n\left(n+1\right)\left|u_n-u_{n+1}\right|^2
+
R\left|u_R\right|^2
\le
\frac12\left\|\mathbf u\right\|_{\C^R}^2
+
2\mathcal D_R\left(\mathbf u\right).
\end{aligned}
\]
Thus $\mathcal D_R(\mathbf u)\ge\|\mathbf u\|_{\C^R}^2/4$.
The Hermitian part of $\mathbf H_{a,R}$ depends affinely on
$\sigma$. At $\sigma=0$, its quadratic form is
\[
-\Re\sum_{n=1}^{R-1}\overline{u_n}u_{n+1}
-R\left|u_R\right|^2
\le
\left\|\mathbf u\right\|_{\C^R}^2.
\]
At $\sigma=1$, it is
$-\mathcal D_R(\mathbf u)
-\Re\sum_{n=1}^{R-1}\overline{u_n}u_{n+1}$.
Expanding the adjacent products, and using
$|u_1|^2\le\mathcal D_R(\mathbf u)$ together with
\[
\left|u_R\right|^2
+
\sum_{n=1}^{R-1}\left|u_n-u_{n+1}\right|^2
\le
\frac12\mathcal D_R\left(\mathbf u\right),
\]
gives
\[
\begin{aligned}
-\mathcal D_R\left(\mathbf u\right)
-\Re\sum_{n=1}^{R-1}\overline{u_n}u_{n+1}
&\le
-\left\|\mathbf u\right\|_{\C^R}^2
-\frac14\mathcal D_R\left(\mathbf u\right)
\le
-\frac{17}{16}\left\|\mathbf u\right\|_{\C^R}^2.
\end{aligned}
\]
Interpolation between $\sigma=0$ and $\sigma=1$ proves
\eqref{eq:nonzero-energy}.

For $0<t\le x$, let $\mathcal V_{a,R}(x,t)$ be the solution
operator on $\C^R$, defined by
\[
\partial_x\mathcal V_{a,R}\left(x,t\right)
=\left(\ii\mathcal N+\frac{\mathbf H_{a,R}}x\right)
\mathcal V_{a,R}\left(x,t\right),
\qquad
\mathcal V_{a,R}\left(t,t\right)=\mathbf I.
\]
Thus $\mathcal V_{a,R}(x,t)$ maps the value of a solution at
$t$ to its value at $x$. For fixed $R,x,t$, it is holomorphic
in $a\in\Omega$, by holomorphic dependence on parameters for
finite-dimensional linear differential equations.

To bound its operator norm, fix $\mathbf u_t\in\C^R$ and put
$\mathbf u(x)=\mathcal V_{a,R}(x,t)\mathbf u_t$.
Since $\ii\mathcal N$ is skew-adjoint,
\eqref{eq:nonzero-energy} gives
\[
\frac{\dd}{\dd x}\left\|\mathbf u\left(x\right)\right\|_{\C^R}^2
=\frac2x\Re\left(\mathbf u\left(x\right),
\mathbf H_{a,R}\mathbf u\left(x\right)\right)_{\C^R}
\le\frac{2\kappa}{x}
\left\|\mathbf u\left(x\right)\right\|_{\C^R}^2.
\]
Integrating and taking the supremum over
$\|\mathbf u_t\|_{\C^R}=1$ yields
\begin{equation}\label{eq:finite-evolution-bound}
\left\|\mathcal V_{a,R}\left(x,t\right)\right\|_{\C^R\to\C^R}
\le\left(\frac xt\right)^\kappa,
\qquad 0<t\le x.
\end{equation}

We now construct the finite approximation for every $x>0$:
\begin{equation}\label{eq:finite-entrance}
\begin{gathered}
\mathbf z_R\left(x;a\right)\defeq
\mathbf P_R\boldsymbol\psi\left(a\right)
+\ii\int_0^x\mathcal V_{a,R}\left(x,t\right)
\mathbf P_R\mathcal N\boldsymbol\psi\left(a\right)\,\dd t,
\\
\left\|\mathbf z_R\left(x;a\right)
-\mathbf P_R\boldsymbol\psi\left(a\right)\right\|_{\C^R}
\le\frac{x}{1-\kappa}
\left\|\mathcal N\boldsymbol\psi\left(a\right)\right\|_{\ell^2\left(\N\right)}.
\end{gathered}
\end{equation}
The bound follows by integrating $(x/t)^\kappa$ over $(0,x)$.
On compact subsets of $\Omega$, $1-\kappa=33\sigma/16$ has a
positive lower bound. The integral therefore converges locally
uniformly in $a$, so $\mathbf z_R(x;a)$ is holomorphic on $\Omega$.
By \eqref{eq:finite-entrance-balance}, differentiation gives
\begin{equation}\label{eq:finite-weighted-system}
\partial_x\mathbf z_R\left(x;a\right)
=\left(\ii\mathcal N+\frac{\mathbf H_{a,R}}x\right)
\mathbf z_R\left(x;a\right)
+\frac{\mathbf b\left(a\right)}x,
\qquad x>0.
\end{equation}

\textbf{Step 2. Uniform estimates at infinity.}
The first-mode amplitude equation in Step~3 will have a remainder
bounded by a constant times
$x^{2\sigma-2}\|\mathbf z_R(x;a)\|_{\C^R}$.
We now prove that this quantity is integrable, with tails tending
to zero uniformly in $R$ and locally uniformly in $a$.
Remove the explicit forced term by setting
\[
\mathbf c\left(a\right)\defeq
-\ii\psi_0\left(a\right)\mathbf e_1,
\qquad
\mathbf w_R\left(x;a\right)\defeq
\mathbf z_R\left(x;a\right)-\frac{\mathbf c\left(a\right)}x.
\]
Here $\mathbf c$ is regarded as a vector in $\C^R$.
Since $\ii\mathcal N\mathbf c+\mathbf b=0$, we have
\[
\partial_x\mathbf w_R
=
\left(\ii\mathcal N+\frac{\mathbf H_{a,R}}x\right)\mathbf w_R
+
\frac{\left(\mathbf H_{a,R}+\mathbf I\right)\mathbf c}{x^2}.
\]
The vector $(\mathbf H_{a,R}+\mathbf I)\mathbf c$ is supported
on the first two coordinates, and
\[
\left\|
\left(\mathbf H_{a,R}+\mathbf I\right)\mathbf c
\right\|_{\C^R}
\le C
\]
uniformly in $R\ge3$ and $a\in\Omega$, since $|p|<2$ and
$|\psi_0(a)|\le1$.
Variation of constants gives
\begin{equation}\label{eq:complex-duhamel-bound}
\left\|\mathbf w_R\left(x;a\right)\right\|_{\C^R}
\le
x^\kappa
\left\|\mathbf w_R\left(1;a\right)\right\|_{\C^R}
+
C\int_1^x\left(\frac xt\right)^\kappa t^{-2}\,\dd t,
\qquad x\ge1.
\end{equation}
Since $1-2\sigma-\kappa=\sigma/16$, Tonelli's theorem yields
\begin{equation}\label{eq:integrated-complex-bound}
\begin{aligned}
\int_1^\infty
x^{2\sigma-2}
\left\|\mathbf z_R\left(x;a\right)\right\|_{\C^R}\,\dd x
&\le
\frac{16}{\sigma}
\left[
\left\|\mathbf w_R\left(1;a\right)\right\|_{\C^R}
+\frac C{2-2\sigma}
\right]
+\frac C{2-2\sigma},
\\
\left\|\mathbf w_R\left(1;a\right)\right\|_{\C^R}
&\le
C\left[1+\frac{1+\left|a\right|}{\sigma}\right].
\end{aligned}
\end{equation}
The second inequality follows from
\eqref{eq:finite-entrance} and
\eqref{eq:entrance-fourier-bounds}.

To obtain uniform tails, fix a nonempty compact set $K\subset\Omega$
and put
\[
m_K\defeq\min_{a\in K}\sigma\left(a\right)>0,
\qquad
M_K\defeq\max_{a\in K}\sigma\left(a\right)<1.
\]
For every real $\kappa$ and $x\ge1$,
\[
\begin{aligned}
\int_1^x\left(\frac xt\right)^\kappa t^{-2}\,\dd t
&=\int_1^x x^\kappa t^{-\kappa-1}\,\frac{\dd t}{t}
\le\log x\,\max\left(x^\kappa,x^{-1}\right).
\end{aligned}
\]
Thus \eqref{eq:complex-duhamel-bound}, the locally uniform bound
on $\mathbf w_R(1;a)$, and $\mathbf z_R=\mathbf w_R+\mathbf c/x$
give
\[
\left\|\mathbf z_R\left(x;a\right)\right\|_{\C^R}
\le C_K\left(1+\log x\right)
\left(x^{\kappa\left(a\right)}+x^{-1}\right),
\qquad x\ge1,\quad a\in K.
\]
Using $\kappa(a)+2\sigma(a)=1-\sigma(a)/16$ before taking
extrema over $K$, we obtain
\[
x^{2\sigma\left(a\right)-2}
\left\|\mathbf z_R\left(x;a\right)\right\|_{\C^R}
\le C_K\left(1+\log x\right)
\left(x^{-1-m_K/16}+x^{-3+2M_K}\right).
\]
Both powers on the right have exponent strictly less than $-1$.
This is therefore an integrable majorant independent of $R$ and
$a\in K$, and consequently
\begin{equation}\label{eq:complex-uniform-tails}
\lim_{L\to\infty}
\sup_{\substack{R\ge3\\a\in K}}
\int_L^\infty x^{2\sigma\left(a\right)-2}
\left\|\mathbf z_R\left(x;a\right)\right\|_{\C^R}\,\dd x
=0.
\end{equation}

\textbf{Step 3. Holomorphic finite-system amplitudes.}
We now extract the first-mode amplitude and prove its holomorphy
and polynomial bound, uniformly in $R$. To cancel the leading
interaction with the second mode, introduce the auxiliary expression
\begin{equation}\label{eq:weighted-first-mode}
\widetilde z_{R,1}\left(x;a\right)\defeq
z_{R,1}\left(x;a\right)
+
\frac{\ii}{x}
\left[
\psi_0\left(a\right)
+\left(1+p\right)z_{R,2}\left(x;a\right)
\right].
\end{equation}
This is a first-mode Harris--Lutz correction \cite{HarrisLutz}.
The first two rows of \eqref{eq:finite-weighted-system},
which are unaffected by the boundary modification because
$R\ge3$, give
\begin{equation}\label{eq:weighted-amplitude-derivative}
\begin{gathered}
\frac{\dd}{\dd x}
\left(
\ee^{-\ii x}x^p\widetilde z_{R,1}\left(x;a\right)
\right)
=
\ee^{-\ii x}x^{p-2}
\left[
\ii\left(p-1\right)\psi_0\left(a\right)
+\mathcal J_R\left(x;a\right)
\right],
\\
\mathcal J_R\left(x;a\right)\defeq
\ii\left(1+p\right)
\left[
\left(p-2\right)z_{R,1}\left(x;a\right)
-\left(3p+1\right)z_{R,2}\left(x;a\right)
+\left(3p+2\right)z_{R,3}\left(x;a\right)
\right].
\end{gathered}
\end{equation}
Throughout, $x^p=\exp(p\log x)$ uses the real logarithm.

We integrate the two terms on the right separately.
For the explicit term, one integration by parts gives
\[
\int_1^\infty\ee^{-\ii t}t^{p-2}\,\dd t
=
-\ii\ee^{-\ii}
-\ii\left(p-2\right)
\int_1^\infty\ee^{-\ii t}t^{p-3}\,\dd t.
\]
The last integral converges absolutely locally uniformly on
$\Omega$. Indeed, for every compact $K\subset\Omega$ there
is $\delta_K>0$ such that $\Re p\le2-\delta_K$ on $K$,
and hence $|t^{p-3}|\le t^{-1-\delta_K}$ for $t\ge1$.
The same integration by parts, with lower endpoint $x$,
shows that the tail of the explicit oscillatory integral
tends to zero locally uniformly in $a$.

For the remaining term,
\[
\left|
\ee^{-\ii t}t^{p-2}\mathcal J_R\left(t;a\right)
\right|
\le
C_K t^{2\sigma\left(a\right)-2}
\left\|\mathbf z_R\left(t;a\right)\right\|_{\C^R},
\qquad a\in K.
\]
The integrated bound \eqref{eq:integrated-complex-bound} and
\eqref{eq:complex-uniform-tails} therefore
show that this term has an absolutely convergent integral,
with tails tending to zero uniformly in $R$ and $a\in K$.
Consequently the limits
\[
\mathcal A_R\left(a\right)\defeq
\lim_{x\to\infty}
\ee^{-\ii x}x^p\widetilde z_{R,1}\left(x;a\right)
\]
exist and satisfy
\begin{equation}\label{eq:finite-amplitude-tails}
\lim_{x\to\infty}
\sup_{\substack{R\ge3\\a\in K}}
\left|
\mathcal A_R\left(a\right)
-\ee^{-\ii x}x^p\widetilde z_{R,1}\left(x;a\right)
\right|
=0
\end{equation}
for every compact $K\subset\Omega$.
Each finite-$x$ expression is holomorphic in $a$, so its locally
uniform limit $\mathcal A_R$ is holomorphic on $\Omega$.
The bounds at $x=1$ and the integrated estimates also give
$\sup_{R\ge3,\,a\in K}|\mathcal A_R(a)|<\infty$.

For the polynomial bound, on $\Re a>3$ we have
\[
\sigma<\frac13,
\qquad
\sigma^{-1}=\frac{\left|a\right|^2}{\Re a}
\le\frac{\left|a\right|^2}{3}.
\]
Thus \eqref{eq:integrated-complex-bound} gives, uniformly in $R$,
\[
\left\|\mathbf w_R\left(1;a\right)\right\|_{\C^R}
\le C\left(1+\left|a\right|\right)^3,
\qquad
\int_1^\infty
x^{2\sigma-2}
\left\|\mathbf z_R\left(x;a\right)\right\|_{\C^R}\,\dd x
\le C\left(1+\left|a\right|\right)^5.
\]
The initial value $\widetilde z_{R,1}(1;a)$ is therefore
$O((1+|a|)^3)$.
The explicit oscillatory integral is uniformly bounded on
$\Re a>3$, by its integration-by-parts formula, and the
$\mathcal J_R$ integral is bounded by the last display.
Integrating \eqref{eq:weighted-amplitude-derivative} consequently
gives
\begin{equation}\label{eq:finite-amplitude-growth}
\left|\mathcal A_R\left(a\right)\right|
\le C\left(1+\left|a\right|\right)^5,
\qquad \Re a>3,
\end{equation}
uniformly in $R$.

\textbf{Step 4. Identification with the physical amplitude and analytic continuation.}
It remains to identify the limit of $\mathcal A_R(a)$ for real
$a>1$. We first compare the finite systems with the physical
moments at fixed $x$, and then use the uniform amplitude tails
to pass to the large-$x$ limit.
Fix $a>1$, so $0<p<2$, and put
\[
h\left(x\right)\defeq1-\Re\mathfrak q_1\left(x;a\right)
=\E\left[1-\cos\mathsf\alpha_0\left(x;a\right)\right].
\]
We first obtain a quantitative entrance estimate for the physical
moments. Taking real parts in the first equation of
\eqref{eq:q-hierarchy} gives
\[
\begin{aligned}
h'\left(x\right)
={}&\E\left[\sin\mathsf\alpha_0\left(x;a\right)\right]
-\frac1x\E\left[
\left(1-\cos\mathsf\alpha_0\left(x;a\right)\right)
\left(2+\left(2-p\right)\cos\mathsf\alpha_0\left(x;a\right)\right)
\right].
\end{aligned}
\]
Since $0<p<2$, the second factor in the last expectation is
at least $p$. Also, $\sin^2\alpha\le2(1-\cos\alpha)$.
Cauchy--Schwarz therefore gives
\[
h'\left(x\right)
\le\sqrt{2h\left(x\right)}-\frac p x h\left(x\right)
\le\sqrt{2h\left(x\right)}.
\]
Here $h\ge0$ and $h(x)\to0$ as $x\downarrow0$.
To integrate the inequality at possible zeros of $h$, fix
$\eta>0$ and observe that
$\frac{\dd}{\dd x}\sqrt{h(x)+\eta}\le1/\sqrt2$.
Integrating from $\varepsilon$ to $x$, then letting
$\eta\downarrow0$ and $\varepsilon\downarrow0$, yields
\begin{equation}\label{eq:physical-entrance-rate}
h\left(x\right)\le\frac{x^2}{2},
\qquad x>0.
\end{equation}

For $|z|=1$ and $n\ge1$, telescoping gives
$|z^n-1|\le n|z-1|$.
Together with \eqref{eq:physical-entrance-rate}, this implies
\[
\begin{aligned}
\left|\mathfrak q_n\left(x;a\right)-1\right|
&\le n\E\left[
\left|\ee^{\ii\mathsf\alpha_0\left(x;a\right)}-1\right|
\right]
\le n\sqrt{2h\left(x\right)}\le nx,
\\
\left|\mathfrak q_{n+1}\left(x;a\right)
-\mathfrak q_n\left(x;a\right)\right|
&\le\E\left[
\left|\ee^{\ii\mathsf\alpha_0\left(x;a\right)}-1\right|
\right]
\le x.
\end{aligned}
\]
The first estimate and \eqref{eq:entrance-fourier-bounds} give
\[
\left\|\mathbf z\left(x;a\right)
-\boldsymbol\psi\left(a\right)\right\|_{\ell^2\left(\N\right)}
\le
x\left\|\mathcal N\boldsymbol\psi\left(a\right)\right\|_{\ell^2\left(\N\right)}.
\]
Thus both the finite approximation and the physical weighted
moments approach their entrance vectors at rate $O(x)$.

Define
\[
\mathbf d_R\left(x;a\right)\defeq
\mathbf z_R\left(x;a\right)-\mathbf P_R\mathbf z\left(x;a\right).
\]
The preceding estimate and \eqref{eq:finite-entrance} show that
$\|\mathbf d_R(\varepsilon;a)\|_{\C^R}\le C_a\varepsilon$,
uniformly in $R$. The equations agree in their first $R-1$
rows. In the last row, the identity
$(R-a)\psi_R=(R+a+1)\psi_{R+1}$ gives
\[
\begin{aligned}
\partial_x\mathbf d_R\left(x;a\right)
={}&\left(\ii\mathcal N+\frac{\mathbf H_{a,R}}x\right)
\mathbf d_R\left(x;a\right)
+\frac{R\left(R+a+1\right)}{ax}\psi_{R+1}\left(a\right)
\left[
\mathfrak q_R\left(x;a\right)-\mathfrak q_{R+1}\left(x;a\right)
\right]\mathbf e_R.
\end{aligned}
\]
For fixed real $a>1$, the Gamma formula gives
$|\psi_n(a)|\le C_a(1+n)^{-2a-1}$.
The adjacent-moment estimate above therefore bounds the norm
of the additional forcing by $C_aR^{1-2a}$ for every $x>0$.
At integer $a=m$, this forcing is zero once $R\ge m$.

Fix $X>0$ and $0<\varepsilon<X$.
Variation of constants and \eqref{eq:finite-evolution-bound} give
\[
\begin{aligned}
\left\|\mathbf d_R\left(X;a\right)\right\|_{\C^R}
&\le C_aX^\kappa\varepsilon^{1-\kappa}
+C_aR^{1-2a}\int_\varepsilon^X
\left(\frac Xt\right)^\kappa\,\dd t
\le C_aX^\kappa\varepsilon^{1-\kappa}
+\frac{C_aX}{1-\kappa}R^{1-2a}.
\end{aligned}
\]
Letting $\varepsilon\downarrow0$, and using $\kappa<1$,
shows that $\|\mathbf d_R(X;a)\|_{\C^R}\le C_aXR^{1-2a}$.
In particular, for every fixed $n\ge1$ and $x>0$,
\begin{equation}\label{eq:finite-physical-coordinate-limit}
z_{R,n}\left(x;a\right)
\longrightarrow
\psi_n\left(a\right)\mathfrak q_n\left(x;a\right)
\qquad\text{as }R\to\infty.
\end{equation}

For the physical moments, put
\[
\widetilde z_1\left(x;a\right)\defeq
z_1\left(x;a\right)
+\frac{\ii}{x}
\left[
\psi_0\left(a\right)+\left(1+p\right)z_2\left(x;a\right)
\right].
\]
Since $\psi_1=-a\psi_0/(a+1)=\psi_0t_1$, we have
\[
\begin{aligned}
\widetilde z_1\left(x;a\right)
={}&\psi_1\left(a\right)
\left(
\mathfrak q_1\left(x;a\right)-\frac{\ii\left(a+1\right)}{ax}
\right)
+\frac{\ii\left(1+p\right)\psi_2\left(a\right)}{x}
\mathfrak q_2\left(x;a\right).
\end{aligned}
\]
The amplitude definition \eqref{eq:amplitude-definition} and the
estimate $\mathfrak q_2(x;a)=o(x^{-p})$ established in the proof of
Proposition~\ref{prop:real-asymptotics} therefore give
\[
\ee^{-\ii x}x^p\widetilde z_1\left(x;a\right)
\longrightarrow\psi_0\left(a\right)\mathfrak U\left(a\right)
\qquad\text{as }x\to\infty.
\]
On the other hand, \eqref{eq:finite-physical-coordinate-limit}
gives $\widetilde z_{R,1}(x;a)\to\widetilde z_1(x;a)$ for
each fixed $x>0$. The uniform tail estimate
\eqref{eq:finite-amplitude-tails} makes these limits compatible:
\[
\begin{aligned}
\limsup_{R\to\infty}
\left|
\mathcal A_R\left(a\right)
-\psi_0\left(a\right)\mathfrak U\left(a\right)
\right|
\le{}&
\sup_{R\ge3}
\left|
\mathcal A_R\left(a\right)
-\ee^{-\ii x}x^p\widetilde z_{R,1}\left(x;a\right)
\right|
\\
&+
\left|
\ee^{-\ii x}x^p\widetilde z_1\left(x;a\right)
-\psi_0\left(a\right)\mathfrak U\left(a\right)
\right|.
\end{aligned}
\]
Both terms tend to zero as $x\to\infty$, proving
$\mathcal A_R(a)\to\psi_0(a)\mathfrak U(a)$ for every real $a>1$.

The holomorphic functions $\mathcal A_R$ are locally uniformly
bounded on $\Omega$ and converge on the interval $(1,\infty)$,
which has accumulation points inside $\Omega$.
Vitali's theorem \cite[Theorem~2.1]{ArendtNikolski} therefore
gives locally uniform convergence on $\Omega$ to a holomorphic
function $\mathcal A$ extending $\psi_0(a)\mathfrak U(a)$.
Passing to the limit in \eqref{eq:finite-amplitude-growth}
gives the asserted polynomial bound.
\end{proof}
The integer-parameter asymptotics now determine the oscillatory
coefficient throughout the required range.
\begin{prop}\label{thm:connection}
For every real $a>1$,
\begin{equation}\label{eq:amplitude-evaluation}
 \mathfrak U\left(a\right)=\Gamma\left(1+1/a\right)^2.
\end{equation}
\end{prop}

\begin{proof}
For $\beta=2m$, Forrester's Selberg integral is the pointwise
scaling limit of the finite circular-ensemble pair correlations
\cite{Forrester1992,ForresterBook}.
Proposition~2.15 of \cite{AssiotisNajnudel} identifies their limit
against continuous compactly supported test functions with the
$\Sine_{2m}$ pair correlation. The uniform bound (21) in its proof
allows dominated convergence, so the Selberg expression agrees
with this pair correlation almost everywhere. Equality holds
everywhere by continuity of both sides, using
\cite[Theorem~1.8]{AssiotisNajnudel}.
The normalisation in \cite{AssiotisNajnudel} has intensity
$1/(2\pi)$, so its pair correlation at $(0,x)$ is
$\mathcal R_m(x)/(4\pi^2)$.
Forrester's asymptotics
\cite[Proposition~13.2.4, pp.~604--605]{ForresterBook}
therefore give, in the present normalisation,
\[
 \mathcal R_m\left(x\right)
 =1+2\Gamma\left(1+1/m\right)^2\cos\left(x\right)x^{-2/m}
       +o\left(x^{-2/m}\right),\qquad m\ge4.
\]
Comparison with Proposition~\ref{prop:real-asymptotics}, along
$x=2\pi k$ and $x=2\pi k+\pi/2$, gives
$\mathfrak U(m)=\Gamma(1+1/m)^2$.

For real $a>1$, put
\[
 \Delta\left(a\right)\defeq
 \psi_0\left(a\right)\left[
       \mathfrak U\left(a\right)-\Gamma\left(1+1/a\right)^2\right].
\]
Lemma~\ref{lem:amplitude-continuation} extends this function
holomorphically to $\Omega$. The Fourier integral gives
$|\psi_0(a)|\le1$, and $\Gamma(1+1/a)^2$ is bounded on $\Omega$.
Consequently
$\Psi(a)=\Delta(a)/(a+1)^M$, with $M$ as in the lemma, is bounded and holomorphic on
$\Omega_3=\{a:\Re a>3\}$, with $\Psi(m)=0$ for all integers
$m\ge4$. The conformal map $w=(a-4)/(a-2)$ sends $\Omega_3$ to
the unit disk and these zeros to $w_m=(m-4)/(m-2)$, for which
\[
 \sum_{m=4}^\infty\left(1-\left|w_m\right|\right)
 =\sum_{m=4}^\infty\frac2{m-2}=\infty.
\]
The Blaschke uniqueness criterion
\cite[discussion preceding Theorem~2.4]{ArendtNikolski} gives
$\Delta=0$ on $\Omega_3$, and the identity theorem extends this
conclusion to $\Omega$. Since $\psi_0(a)>0$ for real $a>1$, this
proves \eqref{eq:amplitude-evaluation}.
\end{proof}

We can finally complete the proof.
\begin{proof}[Proof of Theorem \ref{thm:main}]
For $0<a<1$, Proposition~\ref{prop:real-asymptotics} gives
$\mathcal R_a(x)-1=-2/(ax^2)+o(x^{-2})$.
For $a>1$, the same proposition and Proposition~\ref{thm:connection}
give
\[
 \mathcal R_a\left(x\right)-1
 =2\Gamma\left(1+1/a\right)^2\cos\left(x\right)x^{-2/a}
       +o\left(x^{-2/a}\right).
\]
At $a=1$, the sine-kernel formula gives exactly
\[
 \mathcal R_1\left(x\right)-1
 =\frac{2\left(\cos\left(x\right)-1\right)}{x^2}.
\]
Substituting $a=\beta/2$ and $x=2\pi r$ proves
Theorem~\ref{thm:main}.
\end{proof}

\bibliographystyle{acm}
\bibliography{BKTbib}
\bigskip

\noindent{\sc School of Mathematics, University of Edinburgh, James Clerk Maxwell Building, Peter Guthrie Tait Rd, Edinburgh EH9 3FD, U.K.}\newline
\href{mailto:theo.assiotis@ed.ac.uk}{\small theo.assiotis@ed.ac.uk}

\end{document}